\documentclass[11pt]{article}

\usepackage[margin=1.15in]{geometry}
\usepackage{amsmath,amssymb,amsthm}
\usepackage{mathtools}
\usepackage{authblk}
\usepackage{pgfplots}
\pgfplotsset{compat=1.18}
\usepackage[colorlinks=true,linkcolor=black,citecolor=black]{hyperref}

\newtheorem{theorem}{Theorem}
\newtheorem{proposition}{Proposition}
\newtheorem{lemma}{Lemma}
\newtheorem{corollary}{Corollary}

\theoremstyle{remark}

\newcommand{\R}{\mathbb{R}}
\newcommand{\E}{\mathbb{E}}
\newcommand{\Prob}{\mathbb{P}}
\newcommand{\eps}{\varepsilon}

\title{A stochastic subgradient method with optimal failure exponent}
\date{\today}

\author[1]{Bart P.G.\ van Parys \thanks{bart.van.parys@cwi.nl}}
\affil[1]{CWI Amsterdam}

\begin{document}
\maketitle

\begin{abstract}
Fix a target accuracy $\eps$, a gradient--noise level $s$, and a horizon $N$.
We wish to design algorithms which minimize the probability of observing a suboptimality gap
which exceeds the target accuracy, i.e.,
$\mathcal E_N=-\log\sup_{f,P}\Prob_P(f(x_A)-f_\star\ge\eps)$,
with the noise law $P$ known only to be sub--Gaussian.
We single out a uniformly averaged schedule which is harmonic (of the form $h_k=R^2/(\eps\,(N+m-k))$)
and prove, via an optimized exponential supermartingale argument, that it attains the
optimal exponent $\mathcal E_N^\star = \eps^2N(1+o(1))/(2R^2s^2)$.
Optimality is certified by a matching impossibility result: under Gaussian noise, a
gradient--masking change of measure caps the exponent of every algorithm at the
same leading order.
In a small--noise limit, our setting degenerates into the
adversarial--error model of \cite{gosgens} for subgradient methods.
\end{abstract}

\section{Introduction}\label{sec:intro}

We consider the generic problem of minimizing a nonsmooth convex function given access
only to stochastic subgradients. Let $\mathcal C=\mathcal C(R,L)$ be the class of
nonsmooth convex functions $f:\R^d\to\R$ admitting a minimizer $x_\star$ achieving $f_\star=f(x_\star)$ with
$\|x_0-x_\star\|\le R$ and subgradients bounded by $\|g\|\le L$ everywhere. A black box algorithm
may query, at any point $x$, a stochastic subgradient oracle
\[
  \tilde g\;=\;g+e,\qquad g\in\partial f(x),
\]
and after $N$ queries must return a point $x_A$ with suboptimality gap smaller than a target accuracy $\eps$.
Here the subdifferential
$\partial f$ is set-valued and the returned subgradient is a measurable selection chosen by
the oracle.
We assume that under its law $P$ the noise sequence is adapted and
conditionally sub--Gaussian. That is, with $\mathcal F_k=\sigma(\omega,e_0,\dots,e_{k-1})$ the filtration
generated by the noise sequence $e_0,\dots,e_{k-1}$ together with any internal
randomness $\omega$ available to the algorithm (its random coins, drawn independently
of the noise before the run), we have
\begin{equation}\label{eq:subg}
  \E_P\big[\exp\langle a,e_k\rangle\,\big|\,\mathcal F_k\big]
  \;\le\;\exp\Big(\frac{s^2\|a\|^2}{2}\Big)
  \qquad\text{for all }a\in\R^d ,
\end{equation}
which in particular implies $\E_P[e_k\mid\mathcal F_k]=0$. The stochastic process class contains in particular i.i.d.\
Gaussian noise of variance at most $s^2$ per coordinate. In fact, for
$e_k\sim N(0,s^2I_d)$ the bound \eqref{eq:subg} is tight.

Fix a target accuracy $\eps\in(0,RL)$. The law $P$ is not assumed
known and the object of interest
is the exponent
\[
  \mathcal E_N\;:=\;-\log\,\sup_{f\in\mathcal C}\ \sup_{P}\ \Prob_P\big(f(x_A)-f_\star\ge\eps\big),
\]
of the probability that the algorithm fails to return a solution with suboptimality gap smaller than $\eps$. Here, the inner supremum is running over all adapted laws obeying \eqref{eq:subg}. When a
single law is fixed and clear from the context we drop the dependence on $P$ and
write simply $\E$ and $\Prob$. This
criterion is distinct from the more common expected gap criterion which desires $\E[f(x_A)-f_\star]$ to be small.
Nevertheless, our criterion is not foreign to the classical literature. For instance, \cite{njls} accompany 
their expected suboptimality gap analysis with bounds on {probabilities of large deviations}.
That an expectation is a crude summary of a random loss has since been argued from
several directions. Expectation analysis sees the noise of SGD only through its
variance, missing higher moments \cite{bajovic}, and is blind to tail performance \cite{mvs,bongole}.
Algorithms optimized in expectation may even carry a fragile loss distribution, as
\cite{fanglynn} demonstrate for bandit algorithms; also in our setting this will be observed, as the next paragraph points out.
Holding the target $\eps$ fixed, rather than driving it to zero, is moreover
the natural regime for first--order methods. Indeed, in learning problems the
accuracy worth pursuing is bounded below by the statistical error of the
problem, below which further optimization is futile \cite{bottou}.
In this paper we elevate maximizing the failure exponent $\mathcal E_N$ to the main algorithmic design objective.

\paragraph{Classical subgradient method.}

The classical theory \cite{nemyud,njls} 
bounds the second moment of the noisy subgradients as
$\E[\|g_k+e_k\|^2\mid\mathcal F_k]\le M^2$.
In the present oracle model the iterate $x_k$, and hence the selection $g_k$, is
$\mathcal F_k$--measurable, so the cross term vanishes as $\E[e_k\mid\mathcal F_k]=0$
while \eqref{eq:subg} gives $\E[\|e_k\|^2\mid\mathcal F_k]\le ds^2$, so that we may take
$M^2=L^2+ds^2$. For
the expected gap a constant stepsize
\begin{equation}
  \label{eq:expectation-optimal-steps}
  h_k\equiv R/(M\sqrt{N+1})=R/\sqrt{(N+1)(L^2+ds^2)}
\end{equation}
with uniform averaging $x_A=\sum_{k=0}^{N}x_k/(N+1)$ attains the expected suboptimality gap guarantee
$$\E[f(x_A)-f_\star]\le RM/\sqrt{N+1},$$ minimax--optimal up to constants
under a bounded--variance oracle assumption \cite{nemyud,agarwal}. In our metric such constant stepsize schedules will prove polynomially
sub--optimal (Lemma~\ref{lem:classical}); conversely, the schedule maximizing
$\mathcal E_N$ will turn out to hold its worst--case expected gap at the scale
$\eps$ of the target rather than drive it to zero (Lemma~\ref{lem:price}).
The two criteria are thus fundamentally distinct.

\paragraph{Stochastic subgradient method.}

High--probability bounds do follow from an expected--gap guarantee by a
simple application of Markov's inequality,
\(
  \Prob\big(f(x_A)-f_\star\ge\eps\big)\;\le\; (\E[f(x_A)-f_\star])/\eps
  \;\le\;{RM}/(\eps\sqrt{N+1}).
\)
However these bounds are crude and only establish that $\mathcal E_N$ is logarithmic in $N$.
Instead, \cite{liu} propose a constant stepsize schedule with uniform averaging and establish a high probability suboptimality gap guarantee.
Translated to our setting by matching their bound to our target accuracy
$\eps$, their stepsizes equate to
\begin{equation}\label{eq:liu-steps}
  h_k\equiv\frac{2R^2}{\eps N}
\end{equation}
and their guarantee certifies an exponent
$\mathcal E_N\geq\eps^2N/(48R^2s^2)(1+o(1))$, linear in the horizon $N$. For
this constant schedule our certificate of Corollary~\ref{cor:liu} sharpens the guarantee to
\[
  \mathcal E_N \geq \frac{\eps^2N}{8R^2s^2}\,\big(1+o(1)\big).
\]
Remark that $\mathcal E_N$ depends implicitly on the algorithm considered. We do not make this dependence on the algorithm explicit, but instead will write $\mathcal E^\star_N$ for its supremum over all possible algorithms. In Section~\ref{sec:lower}, we will come to show, via a gradient masking impossibility result, that
\[
  \mathcal E^\star_N = \frac{\eps^2N}{2R^2s^2}\,\big(1+ o(1)\big).
\]
Nevertheless, in Corollary~\ref{cor:constant} we will see that constant step size schedules with uniform averaging leave money on the table as their exponent is at most $\tfrac34\,\mathcal E^\star_N\,(1+o(1))$. To get the optimal exponent $\mathcal E^\star_N$ a larger class of subgradient methods needs to be considered instead.

\paragraph{The stochastic harmonic schedule.}

Both the classical \eqref{eq:expectation-optimal-steps} and stochastic subgradient \eqref{eq:liu-steps} methods are first--order methods which employ deterministic,
fixed step sizes $h_0,\dots,h_{N-1}\ge0$ with iterates
\begin{equation}\label{eq:method}
  x_{k+1}\;=\;x_k-h_k\,(g_k+e_k),\qquad g_k\in\partial f(x_k),
\end{equation}
for $k=0,\dots,N-1$, and output a weighted average
$$\textstyle x_A=\sum_{k=0}^{N}w_kx_k/W$$ with deterministic weights $w_k\ge0$ and normalization
$W=\sum_{k=0}^{N}w_k$.
Within this averaged subgradient class the algorithm design problem reduces to choosing a pair of steps and weights maximizing the exponent $\mathcal E_N$.
The optimum will be identified as the stochastic harmonic schedule using steps
\begin{equation}\label{eq:harmonic}
  h_k=\frac{R^2}{\eps\,(N+m-k)},\qquad k=0,\dots,N-1,\qquad m=\lceil\sqrt{2N}\,\rceil,
\end{equation}
again using uniform averaging. The steps are harmonic in the {remaining} horizon and grow from
$R^2/(\eps N)$ to $R^2/(\eps\sqrt{2N})$ as the horizon runs out (Figure~\ref{fig:schedule}).
An adapted exponential supermartingale, with tilts generated by a certain
recursion, certifies that
it attains the ceiling $\mathcal E^\star_N$ at every fixed noise level.

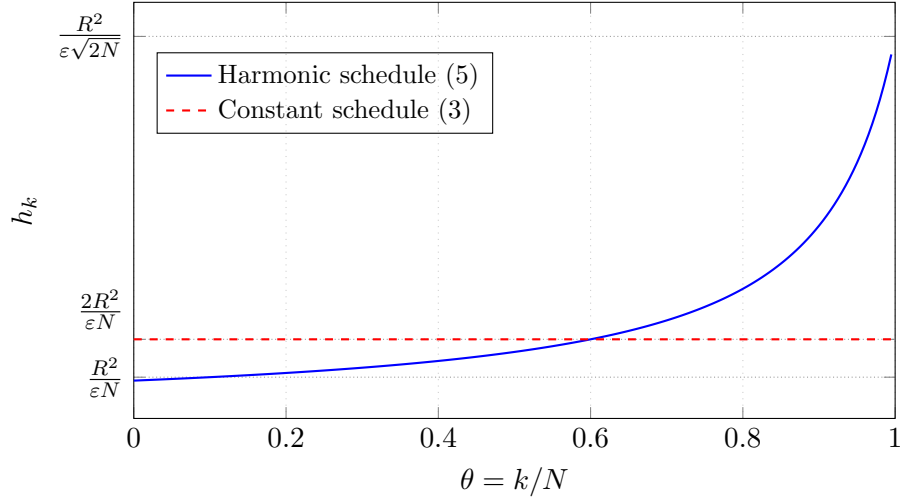
\begin{figure}[ht]
\centering
\begin{tikzpicture}
\begin{axis}[
  width=0.74\textwidth, height=0.45\textwidth,
  xlabel={$\theta=k/N$},
  ylabel={$h_k$},
  legend cell align=left,
  legend style={at={(0.03,0.88)}, anchor=north west, draw=black, fill=none, font=\small},
  xmin=0, xmax=1,
  ymax=1.1,
  xtick={0,0.2,0.4,0.6,0.8,1},
  ytick={0,1},
  yticklabels={$\frac{R^2}{\eps N}$,$\frac{R^2}{\eps\sqrt{2N}}$},
  extra y ticks={0.111111},
  extra y tick labels={$\frac{2R^2}{\eps N}$},
  extra y tick style={yticklabel style={anchor=south east},
    grid=major, major grid style={densely dotted, gray!80}},
  grid=major,
  major x grid style={dotted, gray!40},
  major y grid style={densely dotted, gray!80},
]
\addplot[blue, thick, samples at={0,...,199}] ({x/200}, {(200/(220-x)-1)/9});
\addlegendentry{Harmonic schedule~\eqref{eq:harmonic}}
\addplot[red, thick, dashed, domain=0:1, samples=2] {1/9};
\addlegendentry{Constant schedule~\eqref{eq:liu-steps}}
\end{axis}
\end{tikzpicture}
\caption{The stochastic harmonic schedule \eqref{eq:harmonic} and the
constant schedule \eqref{eq:liu-steps} of \cite{liu} against the normalized
time $\theta=k/N$, for $N=200$. The harmonic
steps remain of order $R^2/(\eps N)$ over most of the horizon and grow
harmonically to order $R^2/(\eps\sqrt{2N})$ over the final
$m=\lceil\sqrt{2N}\,\rceil$ iterations.}
\label{fig:schedule}
\end{figure}

\paragraph{The adversarial harmonic schedule.}
The form of the schedule \eqref{eq:harmonic} has a precedent. In the
adversarial--error model of \cite{gosgens}, where the gradient errors are chosen
adversarially under a total budget constraint
$\sum_{k=0}^{N-1}\|e_k\|^2\le\gamma^2$, a very similar looking schedule
emerges. First note however that averaged subgradient methods output
\[
  \textstyle x_A\;=\;x_0-\sum_{k=0}^{N-1}\alpha_k\,(g_k+e_k),\qquad
  \alpha_k\;=\;h_k\sum_{j>k}w_j\big/W ,
\]
and within the admissible subfamily studied in \cite{gosgens} the adversarial guarantee depends on the schedule
only through its conic combination $\alpha$. Define $u(\sigma)\ge1$ as the unique solution to
\begin{equation}\label{eq:udef}
  \sigma^2\;=\;u(\sigma)^2-1-2\log u(\sigma),
\end{equation}
for $\sigma=\gamma/L$.
The two--sided
analysis of \cite{gosgens} shows that this single curve controls the problem. An
explicitly tuned conic combination $\alpha'$ keeps its suboptimality gap below
${RL}\,u(\sigma)/{\sqrt{N+1}}$ against every error of total budget at most
$L^2\sigma^2$ \cite[Lem.~2]{gosgens}, while no deterministic algorithm whose
trajectory only moves into the negative cone of the observed noisy subgradients \cite{fatkhullin} improves on
this by more than the relative contraction $\delta_N:=\tfrac52\log(N+1)/N$
\cite[Thm.~2 and Cor.~2]{gosgens}. Their tuned combination itself
reads
\[
  \alpha'_k\;=\;\frac{R}{L}\cdot\frac{N-k}{(N+1)^{3/2}}
  \cdot\frac{u(\sigma)\,\xi_N(\sigma)}{u(\sigma)^2-(u(\sigma)^2-1)\frac{k}{N+1}}\,,
  \qquad k<N,
\]
with a correction factor $\xi_N(\sigma)$ obeying $1\le\xi_N(\sigma)\le1+\tfrac2N$ \cite[Lem.~2]{gosgens}.
This tuned combination can be realized with uniform averaging and stepsizes
$h_k=\alpha'_k(N+1)/(N-k)$; as its denominator equals
$\tfrac{u(\sigma)^2}{N+1}\big[N+1-(1-1/u(\sigma)^2)k\big]$, this gives
\[
  h_k\;=\;\frac{R\sqrt{N+1}}{L\,u(\sigma)}\cdot
    \frac{\xi_N(\sigma)}{N+1-(1-1/u(\sigma)^2)k}\,.
\]
Fix a level $\eps'\le\eps$ and match the budget so that
$RL\,u(\sigma)/\sqrt{N+1}=\eps'$; the prefactor is then
$R\sqrt{N+1}/(L\,u(\sigma))=R^2/\eps'$, and the tuned method of \cite{gosgens}
appears as
\begin{equation}\label{eq:gvpavg}
  h_k\;=\;\frac{R^2}{\eps'}\cdot\frac{\xi_N(\sigma)}{N+1-(1-1/u(\sigma)^2)k}\,,
  \qquad k=0,\dots,N-1.
\end{equation}
To distinguish it from
\eqref{eq:harmonic}, we call schedule \eqref{eq:gvpavg} with uniform averaging
the adversarial harmonic schedule matched at $\eps'$. The choice $\eps'=\eps$
matches the target exactly; the freedom to undershoot it slightly is used in
Section~\ref{sec:smallnoise}.

\paragraph{The small--noise limit.}
The resemblance between \eqref{eq:harmonic} and \eqref{eq:gvpavg} is not
coincidental, and can be made formal through a small--noise limit. The results
discussed so far pertain to a regime of a fixed noise level $s$ and a growing
horizon $N$; a complementary regime (horizon $N$ fixed, $s\downarrow0$) is
treated in Section~\ref{sec:smallnoise}. There failure becomes a rare event
and large deviation intuition suggests only the most likely way of failing
needs to be accounted for. An admissible law can place probability
$\exp\big(-(1+o(1))\sum_{k=0}^{N-1}\|e^\star_k\|^2/(2s^2)\big)$ on a given
corruption path $e^\star=(e^\star_0,\dots,e^\star_{N-1})$ and \eqref{eq:subg}
permits no more. The adversarial--error model of \cite{gosgens} hence falls
out of ours as its small--noise rate function, exactly so at fixed horizon
(Proposition~\ref{prop:limit}).

\subsection{Related work}
Closest in regime to the design problem is \cite{bajovic} which derives a genuine large deviation property for the last iterate
in growing iteration count for smooth strongly convex $f$ under the standard
$h_k \equiv h/k$ schedule for a single known noise law.
Also close are \cite{armackihp,armacki},
which pose the fixed--threshold
question for non--convex objective functions in the gradient--norm metric; the latter attains $N/\log N$-speed under decaying
$h_k\equiv h/\sqrt{k+1}$ steps. For strongly convex objective functions, \cite{armackihp} further derives a
large--deviation upper bound on the squared error of the Polyak--Ruppert average, under $h/(k+1)^{r}$ steps with $r\in(1/2,1)$.

In the high--probability literature guarantees hold at a prescribed
confidence level $1-\delta$. Any exponent bound translates into this format at
once, a reliability $1-\delta$ being reached as soon as
$\mathcal E_N\ge\log(1/\delta)$. Read this way, the stochastic harmonic
schedule attains a prescribed confidence at the target $\eps$ within
$2R^2s^2\log(1/\delta)/\eps^2\,(1+o(1))$ queries, no algorithm does so with a
smaller constant (Theorem~\ref{thm:universal}), every constant schedule
requires at least $4/3$ times as many (Lemma~\ref{lem:constant}), and the
matched constant steps \eqref{eq:liu-steps} of \cite{liu} are certified with
four times as many (Corollary~\ref{cor:liu}); the light--tail bounds of
\cite{njls} yield only $\sqrt N$--order exponents at fixed $\eps$. For
non--convex smooth objective
functions \cite{ghadimilan} already reach a $\log(1/\delta)$ dependence, though
by running independent copies of the method and selecting among them in a
post--optimization phase rather than by tuning a single trajectory.
In a related strongly convex setting, \cite{harvey} show a linear dependence on
$\log(1/\delta)$ to be unavoidable for the last iterate and for suffix averages
under the standard $h_k\equiv h/k$ schedule. Horizon--indexed
schedules appear for the last iterate of SGD in \cite{jain} at fixed confidence.

On the impossibility side, the classical oracle
complexity of the class is minimax {in expectation} \cite{nemyud,agarwal}; high--probability variants (minimax quantiles) have recently been developed for estimation \cite{mvs} and
for interactive decision making \cite{bongole}, though at fixed or polynomially small risk
levels, while the local asymptotic minimax theory of \cite{duchiruan} caps unrestricted
estimators at the CLT scale in the smooth regular case.

In Section~\ref{sec:smallnoise} the horizon is fixed and the noise level
$s\downarrow0$. Such small--noise asymptotics originate from the study of small
random perturbations of dynamical systems
\cite{schilder,varadhan,freidlinwentzell}, where the probability of a rare
exceedance decays exponentially in the inverse noise level at a rate fixed by the
most likely way the vanishing perturbation can bring the event about. The
analogue here is the most likely corruption path that drives the method past the
accuracy target.

\subsection*{Notation}

Throughout, $\|\cdot\|$ and $\langle\cdot,\cdot\rangle$ denote the Euclidean
norm and inner product on $\R^d$, $I_d$ the identity matrix, and
$N(\mu,\Sigma)$ the normal law. $\partial f$ is the subdifferential of $f$,
and $\mathrm{sgn}$ the sign function with the convention $\mathrm{sgn}(0):=1$.
We write $(x)_+:=\max\{x,0\}$ for the positive part, $\lceil x\rceil$ for the
ceiling, $\mathbf 1_A$ for the indicator of an event $A$, and $\Phi$ for the
standard normal distribution function, and $H_n:=\sum_{i=1}^ni^{-1}\ge\log n$
for the $n$th harmonic number. Logarithms are natural. Asymptotic
statements $o(\cdot)$ and $O(\cdot)$ refer to $N\to\infty$ at
fixed problem parameters $(\eps,R,L,s,d)$ unless indicated otherwise.

\subsection{Outline}

Section~\ref{sec:classical} shows that the classical designs fall short.
The expectation--optimal constant step schedule reaches exponents only of order
$\sqrt N$ (Lemma~\ref{lem:classical}), no constant schedule can exceed
three quarters of the optimal exponent (Corollary~\ref{cor:constant}), and
the anytime schedule $h/\sqrt{k+1}$ pays a further logarithmic factor
(Lemma~\ref{lem:anytime}). Section~\ref{sec:schedule} then states our
achievability theorem (Theorem~\ref{thm:main}) and proves it by building an
exponential supermartingale certificate (Proposition~\ref{prop:cert}).
Section~\ref{sec:lower} establishes the matching impossibility result: a subgradient--canceling
change of measure caps every algorithm at
$K+\sqrt{6K}+\log 6$ with $K=\eps^2N/(2R^2s^2)$ (Theorem~\ref{thm:universal}).
Section~\ref{sec:smallnoise} closes with an adversarial small--noise limit: at
fixed horizon the small--noise rate is exactly the adversarial guarantee threshold
(Proposition~\ref{prop:limit}), which for the adversarial harmonic schedule
\eqref{eq:gvpavg} is provided by Lemma~\ref{lem:uniform} and shown unimprovable among
methods satisfying the cone condition of \cite{fatkhullin} up to a factor $O(\log N/N)$ in Lemma~\ref{lemma:small-noise-lb}.

\section{Classical schedules}\label{sec:classical}

This section looks at several classical schedule designs in terms of their failure exponent $\mathcal E_N$.
Throughout the paper, the resisting instances realizing the outer supremum in
$\mathcal E_N$ are drawn from the family
$f(x)=\beta\,|x-x_\star|$ with slope $\beta\in(0,L]$ and minimizer $x_\star$
placed, as $\mathcal C$ requires, within $|x_0-x_\star|\le R$ of the start; the
iterates are accordingly univariate, $x_k\in\R$. Every negative result below is
established on a member of this family, paired with a noise law tailored to the
algorithm under attack, and we keep these proofs as uniform as possible so as
to highlight the universal hardness of this type of function. The argument (a certain failure under a tailored law, transferred to the honest noise by a
change of measure) is also the one which will establish our main impossibility
result in Section~\ref{sec:lower}.

We start by looking at schedules with constant steps $h_k\equiv h/\sqrt{N+1}$ with $h>0$ fixed, the
scaling of the expectation--optimal schedule \eqref{eq:expectation-optimal-steps}.
The next lemma shows that the exponent $\mathcal E_N$ is capped here at $O(\sqrt N)$.

\begin{lemma}\label{lem:classical}
Let the noise be independent Gaussian, $e_k\sim N(0,s^2)$, and run the constant schedule
$h_k\equiv h/\sqrt{N+1}$ with uniform averaging, where the scaling $h>0$ is fixed. Then
\[
  \mathcal E_N\;\le\;\frac{4\eps\sqrt{N+1}}{h\,s^2}\,\big(1+o(1)\big),
\]
as $N\to\infty$ at fixed $\eps$ and $s$.
\end{lemma}

\begin{proof}
Write $h_0=h/\sqrt{N+1}$ for the common step.
Let $\gamma>0$ be arbitrary, set $\beta:=\gamma/(2\sqrt{N+1})$,
and take the instance $f(x)=\beta\,|x-x_\star|$ with $x_\star=x_0$, with gradient oracle
returning the subgradient selection
$g(x)=\beta\,\mathrm{sgn}(x-x_\star)\in\partial f(x)$. The instance belongs to $\mathcal C$ whenever
$\beta\le L$, a constraint verified at the end of the proof for the value of $\gamma$
selected there. Let $Q$ be the law of the run under
$e_k=\mu_k+\eta_k$, $k<N$, with the deterministic drift $\mu_k:=-\gamma/\sqrt{N+1}$, of
total energy $E:=\sum_{k=0}^{N-1}\mu_k^2=N\gamma^2/(N+1)\le\gamma^2$, and $\eta_k\sim N(0,s^2)$ i.i.d.

Write $z_k:=x_k-x_0$. The subgradient at $x_k$ is $g_k=\beta\,\mathrm{sgn}(z_k)$,
while under $Q$ the noise is $e_k=-\gamma/\sqrt{N+1}+\eta_k$, so the iteration
\eqref{eq:method} reads
\[
  z_{k+1}\;=\;z_k-h_0\big(g_k+e_k\big)
  \;=\;z_k+h_0\Big(\frac{\gamma}{\sqrt{N+1}}-\beta\,\mathrm{sgn}(z_k)\Big)-h_0\eta_k ,
\]
and since $\beta\,\mathrm{sgn}(z_k)\le\beta=\gamma/(2\sqrt{N+1})$,
\[
  z_{k+1}\;\ge\;z_k+\frac{h_0\gamma}{2\sqrt{N+1}}-h_0\eta_k ,
\]
so by induction $z_k\ge kh_0\gamma/(2\sqrt{N+1})-h_0W_k$ with $W_k:=\sum_{j<k}\eta_j$.
Averaging over the iterates gives
\[
  x_A-x_0\;=\;\frac1{N+1}\sum_{k=0}^{N}z_k
  \;\ge\;\frac{h_0\gamma}{2\sqrt{N+1}}\cdot\frac N2-h_0T ,
  \qquad T:=\frac1{N+1}\sum_{k=0}^{N}W_k .
\]
Consider the event $B:=\{T\le0\}$, of $Q$--probability at least $\tfrac12$ as
$T$ is a centered Gaussian under $Q$. On $B$, multiplying by the slope
$\beta=\gamma/(2\sqrt{N+1})$,
\[
  f(x_A)-f_\star\;\ge\;\beta\,(x_A-x_0)\;\ge\;
  \frac{\gamma^2h_0}{8}\cdot\frac{N}{N+1}\,,
\]
so the explicit choice
\[
  \gamma^2\;:=\;\frac{8\eps}{h_0}\cdot\frac{N+1}N
  \;=\;\frac{8\eps\sqrt{N+1}}{h}\,\Big(1+\frac1N\Big)
\]
makes failure at level $\eps$ certain on $B$. For this choice the deferred
membership constraint $\beta\le L$, i.e., $\gamma^2\le4L^2(N+1)$, reads
$h_0N\ge2\eps/L^2$ and so holds for all $N$ large enough at fixed
$h$, $\eps$ and $s$.

Finally we change measure to the honest law $P$, under which
$e_k\sim N(0,s^2)$ i.i.d. Both runs are driven by the algorithm's random coins
$\omega$ (independent of the noise, with a law $\rho$ common to $P$ and $Q$) together
with the noise stream $(e_0,\dots,e_{N-1})$, with the filtration
$\mathcal F_k=\sigma(\omega,e_0,\dots,e_{k-1})$ of Section~\ref{sec:intro}. Because
each $e_k$ enters only through its conditional law given $\mathcal F_k$, the joint
law of $(\omega,e_0,\dots,e_{N-1})$ disintegrates along the filtration into the coin
marginal times the successive one--step conditional laws of $e_k$ given
$\mathcal F_k$,
\[
  d(\cdot)\;=\;\rho(d\omega)\prod_{k=0}^{N-1}(\cdot)\big(de_k\mid\mathcal F_k\big) .
\]
Under $P$ the conditional of $e_k$ given $\mathcal F_k$ is $N(0,s^2)$ and under $Q$
it is $N(\mu_k,s^2)$, here with the deterministic drift $\mu_k$ as mean; writing
$\varphi_s(u)\propto e^{-u^2/(2s^2)}$ for the $N(0,s^2)$ density these one--step
conditionals have densities $\varphi_s(e_k)$ and $\varphi_s(e_k-\mu_k)$, so the
factorization reads
\[
  dP\;=\;\rho(d\omega)\prod_{k=0}^{N-1}\varphi_s(e_k)\,de_k ,
  \qquad
  dQ\;=\;\rho(d\omega)\prod_{k=0}^{N-1}\varphi_s(e_k-\mu_k)\,de_k .
\]
Using, under $Q$, $e_k=\mu_k+\eta_k$ with
$\eta_k\sim N(0,s^2)$,
\[
  \log\frac{dP}{dQ}
  \;=\;\sum_{k=0}^{N-1}\log\frac{\varphi_s(e_k)}{\varphi_s(e_k-\mu_k)}
  \;=\;\sum_{k=0}^{N-1}\frac{(e_k-\mu_k)^2-e_k^2}{2s^2}
  \;=\;-\frac{\xi}{s^2}-\frac{E}{2s^2},
\]
with $\xi:=\sum_{k=0}^{N-1}\eta_k\mu_k$ and $E\le\gamma^2$ the
energy of the drift introduced above. As the drift is deterministic and the
$\eta_k$ are independent and centered, $\xi$ is under $Q$ a centered Gaussian of
variance
\[
  \mathrm{Var}_Q(\xi)\;=\;\sum_{k=0}^{N-1}\mu_k^2\,\mathrm{Var}_Q(\eta_k)
  \;=\;s^2\sum_{k=0}^{N-1}\mu_k^2\;=\;s^2E ,
\]
so that $\xi/(s\sqrt E)$ is standard normal and, using $E\le\gamma^2$,
$Q(\xi>\gamma s)=1-\Phi(\gamma/\sqrt E)\le1-\Phi(1)\le0.16$. On the event
$\{\xi\le\gamma s\}$ the likelihood ratio obeys
$dP/dQ\ge e^{-\gamma/s-\gamma^2/(2s^2)}$, whence, as failure at level $\eps$ is
certain on $B$,
\begin{align}
  \Prob_P\big(f(x_A)-f_\star\ge\eps\big)
  \;&=\;\E_Q\Big[\mathbf 1_{\{f(x_A)-f_\star\ge\eps\}}\,\frac{dP}{dQ}\Big]\nonumber\\
  \;&\ge\;e^{-\gamma^2/(2s^2)-\gamma/s}\;
  Q\big(B\cap\{\xi\le\gamma s\}\big)
  \;\ge\;\tfrac13\,e^{-\gamma^2/(2s^2)-\gamma/s},\label{eq:changeofmeasure}
\end{align}
the last step by the Fr\'echet lower bound
$Q(B\cap\{\xi\le\gamma s\})\ge Q(B)+Q(\xi\le\gamma s)-1\ge\tfrac12+0.84-1\ge\tfrac13$.
Therefore
$\mathcal E_N\le\gamma^2/(2s^2)+\gamma/s+\log3$, and since
$\gamma/s=o(\gamma^2/s^2)$ as $N\to\infty$, substituting
$\gamma^2=8\eps\sqrt{N+1}/h\,(1+\tfrac1N)$ gives the display.
\end{proof}

We now move to constant steps proposed by \cite{liu} at the scale $R^2/(\eps N)$ used in Equation
\eqref{eq:liu-steps}. As discussed in the introduction, this scaling allows 
reaching exponents of order $\mathcal E_N = O(N)$. The optimal exponent $\mathcal E_N^\star$, however, remains out of
reach for constant schedules.

\begin{lemma}\label{lem:constant}
Let the noise be independent Gaussian, $e_k\sim N(0,s^2)$, and run the constant
schedule $h_k\equiv hR^2/(\eps N)$ with uniform averaging, where $h>0$ is
fixed and $\eps\le\min\{h,\sqrt{h/2}\}\,RL$. Then
\[
  \mathcal E_N\;\le\;\frac{3\big(\sqrt{2h}-1\big)_+^2}{h^2}\cdot
  \frac{\eps^2N}{2R^2s^2}\,\big(1+o(1)\big),
\]
as $N\to\infty$ at fixed $\eps$ and $s$.
\end{lemma}

\begin{proof}
Write $h_0:=hR^2/(\eps N)$ for the common step and
$K:=\eps^2N/(2R^2s^2)$.

\emph{The case $h\le\tfrac12$.} Consider the instance
$f(x)=\beta\,|x-x_\star|$ with minimizer $x_\star:=x_0+R$ and slope $\beta:=\eps/(hR)$,
which is in $\mathcal C$ as $\eps\le hRL$. We consider a noiseless oracle which returns the actual gradient with
selection $g(x)=\beta\,\mathrm{sgn}(x-x_\star)$. The run is then
deterministic and, as long as $x_k<x_\star$, the returned
gradient equals $-\beta$ and the iterate travels
$h_0\beta=R/N$ per step to the right. Indeed, we have $x_k-x_0=kR/N$ exactly, below $R$ at
every queried iterate $k<N$, with resulting
average iterate satisfying $x_A-x_0=R/2$.
Thus,
$f(x_A)-f_\star=\beta R/2=\eps/(2h)\ge\eps$.
Hence $\mathcal E_N=0$.

\emph{Let now $h>\tfrac12$.}
Fix a slope parameter $q>0$ with
$\beta:=\eps/(qR)\le L$, so that
\begin{equation}\label{eq:constant-identities}
  h_0\beta N\;=\;\frac hq\,R,\qquad \frac{\eps}{\beta}\;=\;qR .
\end{equation}
Take the instance
$f(x)=\beta\,|x-x_\star|\in\mathcal C$ with minimizer $x_\star:=x_0+R$, at
distance $R$ from the start. Let $Q$ be the law of the run under $e_k=\mu_k+\eta_k$,
$k<N$, with the linearly decaying drift
$\mu_k:=\lambda\,\frac{N-k}{N}$, $\lambda>0$ selected below, of total energy
\[
  E:=\sum_{k=0}^{N-1}\mu_k^2
  \;=\;\frac{\lambda^2}{N^2}\sum_{k=0}^{N-1}(N-k)^2
  \;=\;\lambda^2\,\frac{(N+1)(2N+1)}{6N}\,,
\]
and $\eta_k\sim N(0,s^2)$ i.i.d., and let the oracle return the selection
$g(x)=\beta\,\mathrm{sgn}(x-x_\star)$ as before. Write
$z_k:=x_k-x_0$ and $y_k:=R-z_k$, the distance still to be covered. The
selection obeys $g_k=\beta\,\mathrm{sgn}(z_k-R)\ge-\beta$, so
the iteration \eqref{eq:method} gives the recursive inequality
\[
  y_{k+1}\;\ge\;y_k-h_0\beta+h_0\lambda\,\frac{N-k}{N}+h_0\eta_k ,
\]
whence by induction $y_k\ge y_{0,k}+h_0W_k$ with $W_k:=\sum_{j<k}\eta_j$
and the noiseless path
\[
  y_{0,k}\;=\;R-h_0\beta k+h_0\lambda\,\frac{k(2N-k+1)}{2N}.
\]
Averaging the noiseless path exactly,
using $\sum_{k=0}^Nk=N(N+1)/2$ and
$\sum_{k=0}^Nk(2N-k+1)=N(N+1)(2N+1)/3$,
\[
  \frac1{N+1}\sum_{k=0}^{N}y_{0,k}
  \;=\;R-\frac{h_0\beta N}2+\frac{h_0\lambda(2N+1)}6 .
\]
Select $\lambda>0$ so that this average exactly meets the failure threshold,
\[
  \frac{h_0\lambda(2N+1)}6\;=\;\frac\eps\beta-R+\frac{h_0\beta N}2
  \;\overset{\eqref{eq:constant-identities}}{=}\;g\,R,
  \qquad g:=q-1+\frac h{2q},
\]
possible whenever $g>0$. With
$\lambda$ so selected, the energy of the
drift equates to
\[
  \gamma^2\;:=\;E
  \;=\;\frac{3g^2R^2}{h_0^2N}\cdot\frac{2(N+1)}{2N+1},
  \qquad
  \frac{3g^2R^2}{2h_0^2Ns^2}\;=\;\frac{3g^2}{h^2}\,K .
\]
\emph{Failure under $Q$.} Consider the event $B:=\{T\ge0\}$ with
$T:=\frac1{N+1}\sum_{k=0}^NW_k$; as $T$ is a centered Gaussian under $Q$,
$Q(B)\ge\tfrac12$. On $B$ the recursive inequality and the selection of
$\lambda$ give
\[
  f(x_A)-f_\star\;=\;\beta\,\Big|\frac1{N+1}\sum_{k=0}^{N}y_k\Big|
  \;\ge\;\beta\,\frac1{N+1}\sum_{k=0}^{N}y_{0,k}\;+\;\beta h_0T
  \;\ge\;\eps :
\]
failure is certain.

\emph{Change of measure.} Using the same change of measure argument as in
Lemma~\ref{lem:classical}, here with energy $E=\gamma^2$, we arrive again at
\eqref{eq:changeofmeasure}.
Assembling, with $\gamma^2/(2s^2)=\tfrac{3g^2}{h^2}K\big(1+\tfrac1{2N+1}\big)$ and
hence $\gamma/s=\tfrac gh\sqrt{6K\big(1+\tfrac1{2N+1}\big)}$,
\begin{equation}\label{eq:constant-master}
  \mathcal E_N\;\le\;\frac{3g^2}{h^2}\,K\Big(1+\frac1{2N+1}\Big)
  +\frac gh\,\sqrt{6K\Big(1+\frac1{2N+1}\Big)}+\log3,
\end{equation}
valid whenever $\beta\le L$ and $g>0$.

Take $q=\sqrt{h/2}$, then $\beta\le L$ is the hypothesis $\eps\le\sqrt{h/2}\,RL$ and
$g=\sqrt{2h}-1>0$, and
\eqref{eq:constant-master} yields the display, its lower--order terms
absorbed in the $(1+o(1))$.
\end{proof}

The prefactor $3(\sqrt{2h}-1)_+^2/h^2$ of the performance bound in Lemma \ref{lem:constant} attains its maximum $3/4$ at $h=2$, precisely the scaling \eqref{eq:liu-steps} suggested by \cite{liu}. The cap of $3/4$ is in fact universal to any constant step--size schedule, as the following result points out.

\begin{corollary}\label{cor:constant}
Let $\eps\le RL$, let the noise be Gaussian, $e_k\sim N(0,s^2)$, and run the constant
schedule $h_k\equiv h_0$ with uniform averaging, where $h_0>0$ is arbitrary
and may depend on $N$ and the problem parameters. Then
\[
  \mathcal E_N\;\le\;\frac34\cdot\frac{\eps^2N}{2R^2s^2}\,\big(1+o(1)\big),
\]
with $o(1)\to0$ as $N\to\infty$ at fixed $\eps$, $R$, $L$ and $s$, uniformly
in $h_0$.
\end{corollary}

\begin{proof}
In the notation of the proof of Lemma~\ref{lem:constant}, write
$h:=\eps Nh_0/R^2$ and take $q=1$ in \eqref{eq:constant-master}. Then,
$\beta=\eps/R\le L$ is exactly the hypothesis $\eps\le RL$ and imposes no
condition on $h_0$, while $g=\tfrac h2>0$ for every $h_0>0$. The
prefactor $3g^2/h^2=\tfrac34$ together with the coefficient $g/h=\tfrac12$
are independent of $h_0$. The master inequality therefore reads
\[
  \mathcal E_N\;\le\;\tfrac34\,K\Big(1+\frac1{2N+1}\Big)
  +\tfrac12\sqrt{6K\Big(1+\frac1{2N+1}\Big)}+\log3
  \;=\;\tfrac34\,K\,\big(1+o(1)\big),
\]
uniformly in $h_0$.
\end{proof}

A final classical design is the anytime schedule $h_k=h/\sqrt{k+1}$ with
$h>0$ fixed. With uniform
averaging it guarantees an expected gap of order $RM\log N/\sqrt N$ at every
horizon simultaneously \cite{shamirzhang}. In our exponent $\mathcal E_N$, this anytime
schedule does not break the $\mathcal E_N=O(\sqrt{N})$ barrier and, compared to the
expectation--optimal step sizes \eqref{eq:expectation-optimal-steps}, incurs a
further $\log(N+1)$ factor.
The argument closely follows that of Lemma~\ref{lem:classical}, the decaying
steps entering only through the two sums $A$ and $S$ of their conic
combination.

\begin{lemma}\label{lem:anytime}
Let the noise be Gaussian, $e_k\sim N(0,s^2)$, and run the anytime
schedule $h_k=h/\sqrt{k+1}$ with uniform averaging. Then, for every $h>0$,
\[
  \mathcal E_N\;\le\;\frac83\cdot\frac{\eps\sqrt{N+1}}{h\,s^2\log(N+1)}\,
  \big(1+o(1)\big),
\]
with $o(1)\to0$ as $N\to\infty$ at fixed $h$, $\eps$ and $s$.
\end{lemma}

\begin{proof}
The conic coefficients of uniform averaging are here
$\alpha_k=h(N-k)/\big(\sqrt{k+1}\,(N+1)\big)$ for $k=0,\dots,N-1$. As
$t\mapsto\big((N+1)-t\big)/\sqrt t$ decreases on $(0,N+1]$, their sum is dominated by the integral,
\[
  A\;:=\;\sum_{k=0}^{N-1}\alpha_k
  \;=\;\frac h{N+1}\sum_{k=0}^{N-1}\frac{(N+1)-(k+1)}{\sqrt{k+1}}
  \;\le\;\frac h{N+1}\int_0^{N}\frac{(N+1)-t}{\sqrt t}\,dt
  \;\le\;\frac43\,h\sqrt{N+1}.
\]
Conversely, rectangles under the curve over $[1,N+1]$ give
\[
  A\;\ge\;\frac h{N+1}\int_1^{N+1}\frac{(N+1)-t}{\sqrt t}\,dt
  \;=\;h\Big(\frac43\sqrt{N+1}-2+\frac2{3(N+1)}\Big)
  \;\ge\;h\Big(\frac43\sqrt{N+1}-2\Big).
\]
The square--sum evaluates in closed form,
\[
  S:=\sum_{k=0}^{N-1}\alpha_k^2
  =\frac{h^2}{(N+1)^2}\sum_{k=0}^{N-1}\frac{(N-k)^2}{k+1}
  =h^2\Big(H_{N+1}-\frac32\Big)+\frac{h^2}{2(N+1)}
  \;\ge\;h^2\Big(\log(N+1)-\frac32\Big).
\]
Set $\beta:=\sqrt{\eps/A}$ so that $\beta\le L$
for all $N$ large enough at fixed $h$ and $\eps$ by the lower bound on $A$
and consider again the instance
$f(x)=\beta\,|x-x_\star|\in\mathcal C$ with $x_\star=x_0$, with
selection $g(x)=\beta\,\mathrm{sgn}(x-x_\star)$ as before. Let
$Q$ be the law of the run under $e_k=\mu_k+\eta_k$, $k<N$, with the
deterministic drift $\mu_k:=-\lambda\alpha_k$, $\lambda>0$ selected below,
of total energy
\[
  E\;:=\;\sum_{k=0}^{N-1}\mu_k^2
  \;=\;\lambda^2\sum_{k=0}^{N-1}\alpha_k^2
  \;=\;\lambda^2S
\]
and $\eta_k\sim N(0,s^2)$ i.i.d. Write $z_k:=x_k-x_0$. The
selection obeys $g_k=\beta\,\mathrm{sgn}(z_k)\le\beta$,
so the iteration \eqref{eq:method} gives the recursive
inequality
\[
  z_{k+1}\;\ge\;z_k-h_k\beta+h_k\lambda\alpha_k-h_k\eta_k ,
\]
whence by induction $z_k\ge z_{0,k}-W_k$ with $W_k:=\sum_{j<k}h_j\eta_j$
and the noiseless path $z_{0,k}:=\sum_{j<k}h_j(\lambda\alpha_j-\beta)$.
Averaging the noiseless path exactly, exchanging the order of summation and
using $h_j(N-j)/(N+1)=\alpha_j$,
\[
  \frac1{N+1}\sum_{k=0}^{N}z_{0,k}
  \;=\;\sum_{j=0}^{N-1}\alpha_j\big(\lambda\alpha_j-\beta\big)
  \;=\;\lambda\sum_{j=0}^{N-1}\alpha_j^2-\beta\sum_{j=0}^{N-1}\alpha_j
  \;=\;\lambda S-\beta A .
\]
Select $\lambda$ so that this average exactly meets the failure threshold,
\[
  \lambda S\;=\;\frac\eps\beta+\beta A\;=\;2\sqrt{\eps A}\,,
\]
a positive drift. With $\lambda$ so selected, the energy of the drift
equates, by the bounds on $A$ and $S$, to
\[
  \gamma^2\;:=\;E\;=\;\frac{4\eps A}S
  \;\le\;\frac{16}3\cdot\frac{\eps\sqrt{N+1}}{h\,\big(\log(N+1)-\tfrac32\big)}\,.
\]

\emph{Failure under $Q$.} Consider the event $B:=\{T\le0\}$ with
$T:=\frac1{N+1}\sum_{k=0}^NW_k=\sum_{j=0}^{N-1}\alpha_j\eta_j$; as $T$ is a
centered Gaussian under $Q$, $Q(B)\ge\tfrac12$. On $B$ the recursive
inequality and the selection of $\lambda$ give
\[
  f(x_A)-f_\star\;=\;\beta\,\Big|\frac1{N+1}\sum_{k=0}^{N}z_k\Big|
  \;\ge\;\beta\big(\lambda S-\beta A\big)-\beta T
  \;\ge\;\eps :
\]
failure is certain.

\emph{Change of measure.} Using the same change of measure argument as in
Lemma~\ref{lem:classical}, here with energy $E=\gamma^2$, we arrive again at
\eqref{eq:changeofmeasure}.
Therefore $\mathcal E_N\le\gamma^2/(2s^2)+\gamma/s+\log3$. At fixed
$h$, $\eps$ and $s$ one has $\gamma^2/s^2\to\infty$ while
$\gamma/s=o(\gamma^2/s^2)$ and $\log(N+1)-\tfrac32=\log(N+1)\,(1-o(1))$, and
the display follows.
\end{proof}

\section{An optimized schedule}\label{sec:schedule}

This section proves the achievability half of the design problem. The stochastic harmonic schedule \eqref{eq:harmonic} of the
introduction attains the exponent $\mathcal E^\star_N = \eps^2N/(2R^2s^2)(1+o(1))$ asymptotically.
Throughout, $\Delta_k:=f(x_k)-f_\star\ge0$ and
$V_k:=\|x_k-x_\star\|^2$.

\begin{theorem}\label{thm:main}
Let $N\ge2$ and run the stochastic harmonic schedule \eqref{eq:harmonic} with uniform
output weights. Then for every noise law obeying \eqref{eq:subg},
\[
  \mathcal E_N\;\ge\;
  \frac{\eps^2}{2R^2s^2}\Big(N+2-m-\frac{2(N+1)}{m}\Big)
  -\Big(\frac{L^2}{s^2}+\frac d2\Big)\log\frac{N+m}{m-1}
  \;=\;\frac{\eps^2N}{2R^2s^2}\,\big(1-o(1)\big).
\]
\end{theorem}

\subsection{An exponential supermartingale certificate}

The proof rests on an exponential supermartingale certificate. A
Chernoff bound first reduces the failure probability to an exponential moment of
the accumulated suboptimality gaps. That moment is then certified by combining
\emph{valid inequalities} of two kinds. The pathwise subgradient
inequalities between the iterates and the minimizer are shared with the adversarial analysis of \cite{gosgens}. Distributional
conditional inequalities on exponential
linear--quadratic moments of the noise complete the stochastic analysis.
The following
lemma is a straightforward conditional variant of the quadratic--form bound of \cite[Rem.~2.3]{hkz}, but we include a short proof to keep the paper self--contained. The martingale
certificate, following ideas in \cite{chunglu,harvey,liu}, combines these
inequalities with free multipliers which we then select judiciously at the end to get the desired exponent guarantee.

\begin{lemma}\label{lem:subg}
Let $e$ satisfy $\E[\exp\langle a,e\rangle\mid\mathcal G]\le\exp(s^2\|a\|^2/2)$ for all
$a\in\R^d$. Then for every $\mathcal G$--measurable vector $u$ and deterministic $b\ge0$ with
$2bs^2<1$,
\begin{equation}\label{eq:mgf}
  \E\big[\exp\big(\langle u,e\rangle+b\|e\|^2\big)\,\big|\,\mathcal G\big]
  \;\le\;\big(1-2bs^2\big)^{-d/2}\exp\!\Big(\frac{s^2\|u\|^2}{2(1-2bs^2)}\Big).
\end{equation}
\end{lemma}

\begin{proof}
Let $g\sim N(0,I_d)$ be independent of everything and recall the Gaussian moment
identity $\E_g[\exp\langle y,g\rangle]=\exp(\|y\|^2/2)$. Applied at
$y=\sqrt{2b}\,e$ with $e$ held fixed, it gives
\[
  \exp\big(b\|e\|^2\big)=\E_g\big[\exp\langle\sqrt{2b}\,g,\,e\rangle\big],
\]
and hence
\[
  \exp\big(\langle u,e\rangle+b\|e\|^2\big)
  =\E_g\big[\exp\langle u+\sqrt{2b}\,g,\,e\rangle\big].
\]
Taking conditional expectations given $\mathcal G$ and swapping the two integrals (justified by Tonelli, the integrand being nonnegative and $g$ independent of $e$ and
$\mathcal G$, so that $\E_g$ is integration against the fixed law $N(0,I_d)$) gives
\[
  \E\big[\exp\big(\langle u,e\rangle+b\|e\|^2\big)\,\big|\,\mathcal G\big]
  =\E\Big[\E_g\big[\exp\langle u+\sqrt{2b}\,g,\,e\rangle\big]\,\Big|\,\mathcal G\Big]
  =\E_g\Big[\E\big[\exp\langle u+\sqrt{2b}\,g,\,e\rangle\,\big|\,\mathcal G\big]\Big].
\]
For each fixed value of $g$ the vector $u+\sqrt{2b}\,g$ is $\mathcal G$--measurable. Hence, we have from sub-Gaussian property
\[
  \E_g\Big[\E\big[\exp\langle u+\sqrt{2b}\,g,\,e\rangle\,\big|\,\mathcal G\big]\Big]
  \;\le\;\E_g\Big[\exp\Big(\tfrac{s^2}2\big\|u+\sqrt{2b}\,g\big\|^2\Big)\Big].
\]
The right--hand side is a standard Gaussian integral, the moment generating
function of a noncentral $\chi^2_d$ \cite[Ch.~29]{jkb}. For
$v\sim N(\mu,\sigma^2I_d)$ and $2t\sigma^2<1$,
\[
  \E\big[\exp\big(t\|v\|^2\big)\big]
  =\big(1-2t\sigma^2\big)^{-d/2}
   \exp\Big(\frac{t\|\mu\|^2}{1-2t\sigma^2}\Big).
\]
Here $v=u+\sqrt{2b}\,g$ has mean $\mu=u$ and covariance $\sigma^2I_d$ with
$\sigma^2=2b$, and $t=s^2/2$, so that $1-2t\sigma^2=1-2bs^2$ and the
right--hand side equals
\[
  \big(1-2bs^2\big)^{-d/2}\exp\Big(\frac{s^2\|u\|^2}{2(1-2bs^2)}\Big).\qedhere
\]
\end{proof}

Based on the previous moment bound a performance certificate is now assembled
from a sequence $\phi$ generated by a tilt recursion.

\begin{proposition}[Martingale certificate]\label{prop:cert}
Let $h_0,\dots,h_{N-1}\ge0$ be any schedule, extended by a free auxiliary stepsize
$h_N\ge0$, and let $\phi_0>0$. Define the tilt sequence
\begin{equation}\label{eq:riccati}
  \frac1{\phi_{k+1}}\;=\;\frac1{\phi_k}+2s^2h_k^2 ,\qquad k=0,\dots,N .
\end{equation}
Let $w_0,\dots,w_N\ge0$ be output weights and $t\ge0$ satisfy $t\,w_k\le2\phi_{k+1}h_k$
for all $k$. Then for every $f\in\mathcal C$ and every noise law obeying \eqref{eq:subg},
\begin{equation}\label{eq:master}
  -\log\Prob\big(f(x_A)-f_\star\ge\eps\big)\;\ge\;
  \eps\,t\,W-\phi_0R^2-L^2\sum_{k=0}^{N}\phi_kh_k^2-\frac d2\log\frac{\phi_0}{\phi_{N+1}}.
\end{equation}
\end{proposition}

\begin{proof}
We extend \eqref{eq:method} to $k=N$ with $e_N:=0$ and the auxiliary point
$x_{N+1}:=x_N-h_Ng_N$, $g_N\in\partial f(x_N)$. All sums run over
$k=0,\dots,N$.

\emph{Step 1: Chernoff reduction.}
By Jensen's inequality applied to the output average,
$f(x_A)-f_\star\le\sum_{k=0}^{N}w_k\Delta_k/W$, so an exponential Markov inequality at threshold
$\eps$ gives
\[
  \Prob\big(f(x_A)-f_\star\ge\eps\big)
  \;\le\;\Prob\Big(t\sum_{k=0}^{N}w_k\Delta_k\ge\eps tW\Big)
  \;\le\;e^{-\eps tW}\,\E\Big[\exp\Big(t\sum_{k=0}^{N}w_k\Delta_k\Big)\Big].
\]
It therefore suffices to certify the exponential moment bound
\begin{equation}\label{eq:moment}
  \E\Big[\exp\Big(t\sum_{k=0}^{N}w_k\Delta_k\Big)\Big]
  \;\le\;\exp\Big(\phi_0R^2+L^2\sum_{k=0}^{N}\phi_kh_k^2+\frac d2\log\frac{\phi_0}{\phi_{N+1}}\Big).
\end{equation}

\emph{Step 2: pathwise inequalities.}
The standard potential--function argument \cite{nemyud,njls} expands the
squared distance to the minimizer, with the noise terms kept exact rather than taken
in expectation. Expanding \eqref{eq:method},
\[
  V_{k+1}\;\le\;V_k-2h_k\langle g_k,x_k-x_\star\rangle+h_k^2\|g_k\|^2
  -2h_k\langle e_k,\,y_k\rangle+h_k^2\|e_k\|^2 ,
  \qquad y_k:=x_k-x_\star-h_kg_k ,
\]
and applying the subgradient inequality between the iterates $x_k$ and $x_\star$,
$\langle g_k,x_k-x_\star\rangle\ge\Delta_k\ge0$, together with $\|g_k\|\le L$,
\begin{equation}\label{eq:onestep}
  V_{k+1}\;\le\;V_k-2h_k\Delta_k+h_k^2L^2-2h_k\langle e_k,y_k\rangle+h_k^2\|e_k\|^2 ,
  \qquad
  \|y_k\|^2\le V_k+h_k^2L^2 ,
\end{equation}
the last bound again by $\langle g_k,x_k-x_\star\rangle\ge0$.

\emph{Step 3: combining the inequalities.}
It remains to combine the pathwise inequalities \eqref{eq:onestep} with the
distributional moment bound of Lemma~\ref{lem:subg} into \eqref{eq:moment}. The combination is
an adapted exponential process carrying the partial sums of the weighted gaps
together with the tilted potential:
\[
  M_k\;:=\;\exp\Big(\textstyle t\sum_{j<k}w_j\Delta_j+\phi_kV_k\Big),
  \qquad k=0,\dots,N+1 .
\]
We show that $M$ is a supermartingale up to a deterministic per--step factor. With
$\beta_k:=\phi_kh_k^2L^2+\tfrac d2\log(\phi_k/\phi_{k+1})$,
\begin{equation}\label{eq:onestepexp}
  \E[M_{k+1}\mid\mathcal F_k]\;\le\;e^{\beta_k}M_k ,\qquad k=0,\dots,N .
\end{equation}
Since the partial sums $t\sum_{j<k}w_j\Delta_j$ are $\mathcal F_k$--measurable
they factor out of $\E[M_{k+1}\mid\mathcal F_k]$, and \eqref{eq:onestepexp} reduces
to
$\E[\exp(tw_k\Delta_k+\phi_{k+1}V_{k+1})\mid\mathcal F_k]\le\exp(\phi_kV_k+\beta_k)$,
which we establish in parts. First, the pathwise inequality \eqref{eq:onestep} gives
\begin{multline*}
  \E\big[\exp\big(tw_k\Delta_k+\phi_{k+1}V_{k+1}\big)\,\big|\,\mathcal F_k\big]\\
  \;\le\;\E\Big[\exp\Big(tw_k\Delta_k+\phi_{k+1}\big(V_k-2h_k\Delta_k+h_k^2L^2
  -2h_k\langle e_k,y_k\rangle+h_k^2\|e_k\|^2\big)\Big)\,\Big|\,\mathcal F_k\Big]\\
  \;=\;\E\Big[\exp\Big(\big(tw_k-2\phi_{k+1}h_k\big)\Delta_k
  +\phi_{k+1}\big(V_k+h_k^2L^2\big)
  +\langle a_k,e_k\rangle+b_k\|e_k\|^2\Big)\,\Big|\,\mathcal F_k\Big],
\end{multline*}
collecting the noise terms with $a_k:=-2\phi_{k+1}h_ky_k$ and
$b_k:=\phi_{k+1}h_k^2$. Since
$(tw_k-2\phi_{k+1}h_k)\Delta_k\le0$ the $\Delta_k$--term may be dropped, and the
$\mathcal F_k$--measurable factors pull out of the conditional expectation. The
right--hand side is at most
\[
  \exp\big(\phi_{k+1}(V_k+h_k^2L^2)\big)\,
  \E\big[\exp\big(\langle a_k,e_k\rangle+b_k\|e_k\|^2\big)\,\big|\,\mathcal F_k\big].
\]
Lemma~\ref{lem:subg} with
$\mathcal G=\mathcal F_k$ and $(u,b)=(a_k,b_k)$, admissible because
\eqref{eq:riccati} gives $1-2b_ks^2=\phi_{k+1}/\phi_k\in(0,1]$, bounds the
remaining conditional expectation, so that
\[
  \E[M_{k+1}\mid\mathcal F_k]\leq
  \exp\big(\phi_{k+1}(V_k+h_k^2L^2)\big)\,
  \big(1-2b_ks^2\big)^{-d/2}
  \exp\Big(\frac{s^2\|a_k\|^2}{2(1-2b_ks^2)}\Big).
\]
Substituting $1-2b_ks^2=\phi_{k+1}/\phi_k$ and
$s^2\|a_k\|^2/(2(1-2b_ks^2))=2\phi_k\phi_{k+1}h_k^2s^2\|y_k\|^2$, this reads
\begin{align*}
  \E[M_{k+1}\mid\mathcal F_k]\leq & 
  \Big(\frac{\phi_k}{\phi_{k+1}}\Big)^{d/2}
  \exp\Big(\phi_{k+1}\big(V_k+h_k^2L^2\big)
                                    +2\phi_k\phi_{k+1}h_k^2s^2\,\|y_k\|^2\Big)\\
  \leq &  \Big(\frac{\phi_k}{\phi_{k+1}}\Big)^{d/2}
  \exp\Big(\big(V_k+h_k^2L^2\big) \big(\phi_{k+1}
                                    +2\phi_k\phi_{k+1}h_k^2s^2\big)\Big)
\end{align*}
using again $\|y_k\|^2\le V_k+h_k^2L^2$. The tilt sequence \eqref{eq:riccati} was chosen so that $\phi_{k+1}+2\phi_k\phi_{k+1}h_k^2s^2=\phi_{k+1}(1+2\phi_kh_k^2s^2)=\phi_k$, and the bound collapses to
\[
  \Big(\frac{\phi_k}{\phi_{k+1}}\Big)^{d/2}
  \exp\big(\phi_k\big(V_k+h_k^2L^2\big)\big)
  \;=\;\exp\big(\phi_kV_k+\beta_k\big).
\]
Iterating \eqref{eq:onestepexp} from $k=N$ down to
$k=0$ via the tower property gives
$\E[M_{N+1}]\le\exp\big(\sum_{k=0}^{N}\beta_k\big)\,\E[M_0]$, with
$M_0=\exp(\phi_0V_0)\le\exp(\phi_0R^2)$ as $V_0\le R^2$, whence
\[
  \E[M_{N+1}]\;\le\;\exp\Big(\phi_0R^2+L^2\sum_{k=0}^{N}\phi_kh_k^2+\frac d2\log\frac{\phi_0}{\phi_{N+1}}\Big)
\]
by the exact telescoping of the $d$--terms in $\sum_k\beta_k$.
Since $\phi_{N+1}V_{N+1}\ge0$ one has
$M_{N+1}\ge\exp\big(t\sum_{k=0}^{N}w_k\Delta_k\big)$, and \eqref{eq:moment} follows.
\end{proof}

Optimizing \eqref{eq:master} over its free multipliers yields a certificate design
problem:
\begin{equation}\label{eq:certopt}
  \mathcal E_N\;\ge\;\sup_{(\phi_0,w,t)}\;
  \Big\{\eps\,t\,W-\phi_0R^2-L^2\sum_{k=0}^{N}\phi_kh_k^2-\frac d2\log\frac{\phi_0}{\phi_{N+1}}\Big\},
\end{equation}
the supremum running over the triples $(\phi_0,w,t)$ admissible in
Proposition~\ref{prop:cert}.
As a first application we evaluate the certificate on the constant schedules of
Section~\ref{sec:classical}.

\begin{corollary}[Constant schedules]\label{cor:liu}
Fix $h>0$ and run the constant schedule $h_k\equiv hR^2/(\eps N)$ with uniform
averaging. Then
\[
  \mathcal E_N\;\ge\;\frac{(\sqrt{2h}-1)_+^2}{h^2}\cdot
  \frac{\eps^2N}{2R^2s^2}\,\big(1+o(1)\big).
\]
\end{corollary}

\begin{proof}
Write $h_0:=hR^2/(\eps N)$ for the common step and take the auxiliary stepsize
$h_N:=h_0$, uniform weights $w_k\equiv1$, any $\phi_0>0$, and the multiplier
$t:=2\phi_{N+1}h_0$. These are admissible in Proposition~\ref{prop:cert} as the
stepsizes, weights and multiplier are nonnegative, and as \eqref{eq:riccati} makes
$1/\phi_k$ nondecreasing we have $\phi_{N+1}\le\phi_{k+1}$, so that
$t\,w_k=2\phi_{N+1}h_0\le2\phi_{k+1}h_k$ for every $k=0,\dots,N$. The
recursion \eqref{eq:riccati} reduces to
$1/\phi_{k+1}=1/\phi_0+2s^2h_0^2(k+1)$. Parametrize $\phi_0=u/(2s^2h_0^2(N+1))$
with $u>0$. The terms of \eqref{eq:master} evaluate separately as follows.

\emph{The gap term.} The solved recursion at $k=N$ gives
$1/\phi_{N+1}=(1+1/u)\,2s^2h_0^2(N+1)$, that is,
$\phi_{N+1}=\frac{u}{1+u}\cdot\frac{1}{2s^2h_0^2(N+1)}$, and with $W=N+1$,
\[
  \eps tW\;=\;2\eps\phi_{N+1}h_0(N+1)\;=\;\frac{\eps}{s^2h_0}\cdot\frac{u}{1+u}.
\]

\emph{The initial tilt.} By the parametrization of $\phi_0$,
\[
  \phi_0R^2\;=\;\frac{R^2u}{2s^2h_0^2(N+1)}.
\]

\emph{The $L^2$--term.} The tilts are nonincreasing, so
\[
  L^2\sum_{k=0}^N\phi_kh_k^2\;=\;L^2h_0^2\sum_{k=0}^N\phi_k
  \;\le\;L^2h_0^2(N+1)\,\phi_0\;=\;\frac{L^2u}{2s^2}.
\]

\emph{The dimension term.} Dividing the parametrization of $\phi_0$ by the
expression for $\phi_{N+1}$ above, $\phi_0/\phi_{N+1}=1+u$, so
\[
  \frac d2\log\frac{\phi_0}{\phi_{N+1}}\;=\;\frac d2\log(1+u).
\]

With $K:=\eps^2N/(2R^2s^2)$ we have
$\eps/(s^2h_0)=2K/h$ and $R^2/(2s^2h_0^2(N+1))\le K/h^2$, so
\[
  \mathcal E_N\;\ge\;\Big(\frac{2}{h}\cdot\frac{u}{1+u}-\frac{u}{h^2}\Big)K
  -\frac{L^2u}{2s^2}-\frac d2\log(1+u).
\]
For $h>\tfrac12$ the bracket is maximized at $u=\sqrt{2h}-1>0$, where it takes
the value $(\sqrt{2h}-1)^2/h^2$, and the trailing terms are $O(1)$ as
$N\to\infty$; for $h\le\tfrac12$ the claim is trivial.
\end{proof}

The certified exponent involves the noise only through the directional scale
$s$ of \eqref{eq:subg}, the dimension entering solely the lower--order
correction, and sits within a factor three of the ceiling of
Lemma~\ref{lem:constant}, whose prefactor is $3(\sqrt{2h}-1)_+^2/h^2$.
At the matched steps of \eqref{eq:liu-steps} ($h=2$) the corollary
certifies $\mathcal E_N\ge\eps^2N/(8R^2s^2)\,(1+o(1))$, improving
significantly on the guarantee of \cite{liu} discussed in the introduction.
The following result shows the certificate optimization problem \eqref{eq:certopt}
has a built--in ceiling, uniform over schedules.

\begin{proposition}\label{prop:selfdual}
For every $N$ and every schedule $h$, the value of the certificate design problem
\eqref{eq:certopt} is at most ${\eps^2(N+1)}/{(2R^2s^2)}.$
\end{proposition}

\begin{proof}
Fix $(\phi_0,w,t)$ admissible in Proposition~\ref{prop:cert}.
Write $\rho_k:=1/\phi_k$, so $\rho_{k+1}=\rho_k+2s^2h_k^2$ and
$h_k=\sqrt{(\rho_{k+1}-\rho_k)/(2s^2)}$. Since $tw_k\le2\phi_{k+1}h_k$,
\[
  \eps tW\;\le\;2\eps\sum_{k=0}^{N}\phi_{k+1}h_k
  \;=\;\frac{2\eps}{\sqrt{2s^2}}\sum_{k=0}^{N}\frac{\sqrt{\rho_{k+1}-\rho_k}}{\rho_{k+1}}
  \;\le\;\frac{2\eps}{\sqrt{2s^2}}\,\sqrt{N+1}\,
  \Big(\sum_{k=0}^{N}\frac{\rho_{k+1}-\rho_k}{\rho_{k+1}^2}\Big)^{1/2}
\]
by Cauchy--Schwarz. As the sequence $\rho$ is positive and nondecreasing, termwise
$(\rho_{k+1}-\rho_k)/\rho_{k+1}^2\le(\rho_{k+1}-\rho_k)/(\rho_k\rho_{k+1})
=1/\rho_k-1/\rho_{k+1}$. The last sum hence telescopes:
$\sum_{k=0}^N(\rho_{k+1}-\rho_k)/\rho_{k+1}^2\le\sum_{k=0}^N(1/\rho_k-1/\rho_{k+1}) = 1/\rho_0-1/\rho_{N+1} \le 1/\rho_0=\phi_0$.
By dropping the (nonnegative) $L^2$-- and $d$--terms, the certified exponent
is hence at most
\[
  2\eps\sqrt{\frac{(N+1)\phi_0}{2s^2}}-\phi_0R^2
  \;\le\;\max_{\varphi>0}\Big(2\eps\sqrt{\tfrac{N+1}{2s^2}}\sqrt\varphi-R^2\varphi\Big)
  \;=\;\frac{\eps^2(N+1)}{2R^2s^2}. \qedhere
\]
\end{proof}

The proof of the main theorem is a direct evaluation of the certificate along the
stochastic harmonic schedule \eqref{eq:harmonic}. The certified exponent matches the ceiling of
Proposition~\ref{prop:selfdual} to leading order, so the stochastic harmonic
schedule asymptotically optimizes the certificate \eqref{eq:certopt} which suggests it is optimal. This suggestion is indeed shown to be correct in Section~\ref{sec:lower} via a matching impossibility result. 

\begin{proof}[Proof of Theorem~\ref{thm:main}]
Choose the auxiliary stepsize of Proposition~\ref{prop:cert} by the same formula,
$h_N:=R^2/(\eps m)$.
Let $\omega:=2s^2R^4/\eps^2$, so that $2s^2h_j^2=\omega/(N+m-j)^2$, and choose
$\phi_0:=(N+m)/\omega$. By \eqref{eq:riccati},
\[
  \frac1{\phi_k}=\frac1{\phi_0}+2s^2\sum_{j<k}h_j^2
  =\frac{\omega}{N+m}+\omega\sum_{j<k}\frac1{(N+m-j)^2}
  =\frac{\omega}{N+m}+\omega\!\!\sum_{i=N+m-k+1}^{N+m}\!\!\frac1{i^2},
\]
and the integral bounds, by monotonicity of $x\mapsto x^{-2}$,
\[
  \frac1{A+1}-\frac1{B+1}=\int_{A+1}^{B+1}\frac{dx}{x^2}
  \;\le\;\sum_{i=A+1}^{B}\frac1{i^2}
  \;\le\;\int_{A}^{B}\frac{dx}{x^2}=\frac1A-\frac1B
\]
applied with $A=N+m-k$ and $B=N+m$ give
\[
  \frac1{\phi_k}\;\le\;\frac{\omega}{N+m}
  +\omega\Big(\frac1{N+m-k}-\frac1{N+m}\Big)\;=\;\frac{\omega}{N+m-k}
\]
and
\[
  \frac1{\phi_k}\;\ge\;\frac{\omega}{N+m}
  +\omega\Big(\frac1{N+m-k+1}-\frac1{N+m+1}\Big)\;\ge\;\frac{\omega}{N+m-k+1},
\]
the last step as $1/(N+m)\ge1/(N+m+1)$. Inverting both bounds gives the
sandwich
\begin{equation}\label{eq:phisandwich}
  \frac{N+m-k}{\omega}\;\le\;\phi_k\;\le\;\frac{N+m-k+1}{\omega},\qquad 0\le k\le N+1 .
\end{equation}
\emph{Leading term.} By \eqref{eq:phisandwich},
$\phi_{k+1}h_k\ge R^2(N+m-k-1)/\big(\eps\omega(N+m-k)\big)
\ge\big(1-1/m\big)R^2/(\eps\omega)$ for all $k\le N$, so the uniform weights
$w_k\equiv1$ admit $t=\big(1-1/m\big)2R^2/(\eps\omega)$, and with
$R^2/\omega=\eps^2/(2s^2R^2)$,
\[
  \eps t(N+1)\;=\;\frac{\eps^2(N+1)}{s^2R^2}\Big(1-\frac1m\Big).
\]
\emph{Remaining terms.} $\phi_0R^2=(N+m)\eps^2/(2s^2R^2)$. For the $L^2$--term, by
\eqref{eq:phisandwich} and $(A+1)/A^2\le2/A$, valid for $A\ge1$ and applied
with $A=N+m-k\ge m\ge1$,
\[
  \sum_{k=0}^{N}\phi_kh_k^2
  \le\frac{R^4}{\eps^2\omega}\sum_{k=0}^{N}\frac{N+m-k+1}{(N+m-k)^2}
  \le\frac{1}{2s^2}\cdot2\!\!\sum_{i=m}^{N+m}\!\frac1i
  \;\le\;\frac1{s^2}\int_{m-1}^{N+m}\frac{dx}{x}
  \;=\;\frac1{s^2}\log\frac{N+m}{m-1}.
\]
For the $d$--term, the sandwich \eqref{eq:phisandwich} at $k=N+1$ gives
$\phi_{N+1}\ge(m-1)/\omega$, while $\phi_0=(N+m)/\omega$, so
$\frac d2\log(\phi_0/\phi_{N+1})\le\frac d2\log\big((N+m)/(m-1)\big)$. Assembling
\eqref{eq:master},
\[
  \mathcal E_N\;\ge\;\frac{\eps^2}{2s^2R^2}\Big(2(N+1)\big(1-\tfrac1m\big)-(N+m)\Big)
  -\Big(\frac{L^2}{s^2}+\frac d2\Big)\log\frac{N+m}{m-1},
\]
which is the display of the theorem.
\end{proof}

The subgradient methods considered in Equation \eqref{eq:method} are unprojected. Nevertheless, the results in this section can be
upgraded trivially to the projected iteration
$x_{k+1}=\Pi_X\big(x_k-h_k(g_k+e_k)\big)$ onto a closed convex set
$X\supseteq\{x:\|x-x_0\|\le R\}$. Indeed, the proof of
Proposition~\ref{prop:cert} uses the update \eqref{eq:method} only through the
squared--distance expansion leading to \eqref{eq:onestep}, which nonexpansiveness of
the Euclidean projection preserves. The certificate
\eqref{eq:master}, the design problem \eqref{eq:certopt},
Corollary~\ref{cor:liu} and Theorem~\ref{thm:main} therefore hold verbatim
for projected stochastic subgradient descent.

\subsection{Expected suboptimality gap}

The stochastic harmonic schedule does exactly what is asked of it: it
optimizes as reliably as possible up to accuracy $\eps$ and, anything beyond
being futile \cite{bottou}, no further. In fact, even in the absence of any
noise its gap on our resisting instance remains $\Omega(\eps)$.

\begin{lemma}\label{lem:price}
Let $N\ge2$ and run the stochastic harmonic schedule \eqref{eq:harmonic} with uniform averaging on the instance
$f(x)=\frac{\eps}{2R}\,|x-x_\star|\in\mathcal C$ with minimizer $x_\star=x_0+R$,
with zero noise. Then
\[
  f(x_A)-f_\star\;\ge\;\frac\eps4 .
\]
\end{lemma}

\begin{proof}
Write $z_k:=x_k-x_0$ and $\beta:=\eps/(2R)\le L$. Each step moves the
iterate by at most $\beta h_j=R/\big(2(N+m-j)\big)$ in the positive direction, so
$z_k\le\frac R2\sum_{j<k}(N+m-j)^{-1}\le\frac R2\int_{0}^{k}\frac{dx}{N+m-x}=\frac R2\log\big((N+m)/(N+m-k)\big)$, and averaging,
\[
  x_A-x_0\;=\;\frac1{N+1}\sum_{k=0}^{N}z_k
  \;\le\;\frac R2\cdot\frac1{N+1}\sum_{k=0}^{N}\log\frac{N+m}{N+m-k}
  \;\le\;\frac R2\cdot\frac1{N+1}\int_0^{N+1}\log\frac{N+m}{N+m-x}\,dx
  \;\le\;\frac R2,
\]
the integral evaluating to $(N+1)-(m-1)\log\big((N+m)/(m-1)\big)\le N+1$.
Hence
$f(x_A)-f_\star=\beta\,(x_\star-x_A)\ge\beta R/2=\eps/4$.
\end{proof}

\section{A gradient masking impossibility result}\label{sec:lower}

We develop an impossibility result capping the failure exponent of any
algorithm, based on independent identically distributed Gaussian noise
$e_k\sim N(0,s^2)$, which saturates \eqref{eq:subg} with equality.
Under this stochastic model the noisy subgradients may mimic uninformative noise carrying no information about the problem instance, and hence no algorithm can hope to make any optimization progress.

\begin{theorem}\label{thm:universal}
  Let $\eps\le RL$ and let the noise be independent and Gaussian, $e_k\sim N(0,s^2)$. Every algorithm
in the oracle model of Section~\ref{sec:intro} satisfies, with
$K:=\eps^2N/(2R^2s^2)$,
\[
  \mathcal E_N\;\le\;K+\sqrt{6K}+\log6 .
\]
\end{theorem}

\begin{proof}
Set $\beta:=\eps/R\le L$ and take the mirrored instances
\[
  f_\pm(x)\;:=\;\beta\,|x-x_\star^\pm|\;\in\;\mathcal C ,
  \qquad x^\pm_\star\;:=\;x_0\pm R ,
\]
with subgradient selections $g_\pm(x)=\beta\,\mathrm{sgn}(x-x^\pm_\star)$,
so that $|g_\pm(x)|=\beta$ everywhere. Let
$Q_\pm$ be the law of the run under the stochastic noise model in which,
conditionally on $\mathcal F_k$, the noise is distributed as
$e_k\sim N(-g_\pm(x_k),s^2)$; equivalently, $e_k=-g_\pm(x_k)+\eta_k$ with
$\eta_k\sim N(0,s^2)$ i.i.d., in other words the oracle only ever returns pure noise,
$\tilde g_k=\eta_k$. Under $Q_+$ and $Q_-$ the observations and
the internal randomness have the same joint law, carrying no information about the
instance, so the output $x_A$ has one law under both. Failure at level $\eps$ on
$f_\pm$ means $|x_A-x^\pm_\star|\ge\eps/\beta=R$; the success
windows $(0,2R)$ and
$(-2R,0)$ for $x_A-x_0$ are disjoint, so the two success probabilities sum to at most one:
\[
  \max\big\{\,Q_+(\text{failure on }f_+),\;Q_-(\text{failure on }f_-)\,\big\}
  \;\ge\;\tfrac12 .
\]
Fix a sign attaining this maximum and write $Q$ and ``failure'' for the
corresponding law and event. It remains to transfer this failure probability to
the honest noise law by a change of measure. Write $P$ for the honest Gaussian
law of the run, with noise $e_k\sim N(0,s^2)$. With $\mathcal F_k=\sigma(\omega,e_0,\dots,e_{k-1})$ the filtration of
Section~\ref{sec:intro}, the iterate $x_k$, and hence also the selection
$g_k:=g_\pm(x_k)$ for the fixed sign, is $\mathcal F_k$--measurable. Exactly as in
the change of measure of Lemma~\ref{lem:classical}, the shared coin marginal cancels
and $\log(dP/dQ)$ is the sum of the conditional density ratios; here the
$Q$--conditional of $e_k$ given $\mathcal F_k$ is $N(-g_k,s^2)$, with
$\mathcal F_k$--measurable mean $-g_k$, so with $\varphi_s$ the $N(0,s^2)$ density
and $e_k+g_k=\eta_k$ under $Q$,
\begin{align*}
  \log\frac{dP}{dQ}
  & =\;\sum_{k=0}^{N-1}\log\frac{\varphi_s(e_k)}{\varphi_s(e_k+g_k)}
  \;=\;\sum_{k=0}^{N-1}\frac{(e_k+g_k)^2-e_k^2}{2s^2}
  \;=\;\sum_{k=0}^{N-1}\frac{2\eta_kg_k-g_k^2}{2s^2}\\
  & =\;\frac1{s^2}\sum_{k=0}^{N-1}\eta_kg_k
  \;-\;\sum_{k=0}^{N-1}\frac{g_k^2}{2s^2}\;=:\;M-K .
\end{align*}
Here $|g_k|=\beta$, so
$K=\beta^2N/(2s^2)=\eps^2N/(2R^2s^2)$ is deterministic. By the tower property of
conditional expectation,
\[
  \E_Q\big[\eta_kg_k\big]
  =\E_Q\Big[\E_Q\big[\eta_kg_k\,\big|\,\mathcal F_k\big]\Big]
  =\E_Q\big[\E_Q[\eta_k\mid\mathcal F_k]\,g_k\big]
  =0
\]
for every $k$, since $g_k$ is $\mathcal F_k$--measurable and $\eta_k$ has
conditional mean zero under $Q$; summing over $k$ gives $\E_Q[M]=0$. The same
conditioning computes the variance. Expanding the square of $M$, a cross term with
$j<k$ vanishes because $\eta_jg_j$ is also
$\mathcal F_k$--measurable,
\[
  \E_Q\big[\eta_jg_j\,\eta_kg_k\big]
  =\E_Q\Big[\E_Q\big[\eta_jg_j\,\eta_kg_k
  \,\big|\,\mathcal F_k\big]\Big]
  =\E_Q\Big[\eta_jg_j\,
  \E_Q[\eta_k\mid\mathcal F_k]\,g_k\Big]
  =0 ,
\]
while a diagonal term follows from the conditional variance
$\E_Q[\eta_k^2\mid\mathcal F_k]=s^2$ and $|g_k|=\beta$,
\[
  \E_Q\big[(\eta_kg_k)^2\big]
  =\E_Q\Big[\E_Q\big[(\eta_kg_k)^2\,\big|\,\mathcal F_k\big]\Big]
  =\E_Q\Big[g_k^2\,\E_Q\big[\eta_k^2\,\big|\,\mathcal F_k\big]\Big]
  =s^2\beta^2 .
\]
Therefore
\[
  \E_Q\big[M^2\big]
  \;=\;\frac1{s^4}\sum_{k=0}^{N-1}\E_Q\big[(\eta_kg_k)^2\big]
  \;=\;\frac{\beta^2N}{s^2}\;=\;2K .
\]
A simple application of Chebyshev's inequality gives $Q(M\le-\sqrt{6K})\le2K/(6K)=\tfrac13$. Hence, finally, we get
\begin{align*}
  \Prob_P\big(\text{failure}\big)\;&=\;
  \E_Q\Big[\mathbf 1_{\{\text{failure}\}}\,\frac{dP}{dQ}\Big]\;\ge\;
  \E_Q\Big[\mathbf 1_{\{\text{failure}\}}\,
  \mathbf 1_{\{M>-\sqrt{6K}\}}\,\frac{dP}{dQ}\Big]\\
  &\ge\;\Big(\tfrac12-\tfrac13\Big)\,e^{-K-\sqrt{6K}}
  \;=\;\tfrac16\,e^{-K-\sqrt{6K}} ,
\end{align*}
where the final inequality bounds $dP/dQ=e^{M-K}\ge e^{-K-\sqrt{6K}}$ on the event
$\{M>-\sqrt{6K}\}$ and applies the Fr\'echet lower bound
$Q\big(\text{failure}\cap\{M>-\sqrt{6K}\}\big)\ge\tfrac12+\tfrac23-1$. Taking
logarithms gives the claim.
\end{proof}

Together with Theorem~\ref{thm:main} this settles the design problem of the
introduction. At fixed accuracy the optimal exponent is
$\mathcal E^\star_N=\eps^2N(1+ o(1))/(2R^2s^2)$ attained by the
stochastic harmonic schedule, exceeded by no algorithm. The constant $K$ matches,
moreover, the built--in ceiling of the certificate
(Proposition~\ref{prop:selfdual}) up to a single query, $\eps^2/(2R^2s^2)$.

\section{The small--noise limit $s\downarrow 0$}\label{sec:smallnoise}

The regimes considered so far hold the noise level $s$ fixed and let the horizon
$N$ grow. In some applications, however, the oracle noise may not be the dominant effect but instead only a small perturbation of an otherwise informative gradient, and the
pertinent question is not whether the method succeeds but how sensitive it is to a small amount of noise. We therefore fix the
horizon $N$ and study the failure exponent as $s\downarrow0$. This is the
classical small--noise regime of the large--deviation theory of \cite{freidlinwentzell} in which
failure becomes a rare event, and its exponent is set by the single most likely way
the vanishing noise can bring that failure about. Which perturbations are
capable of causing failure is an adversarial rather than a stochastic question. The exponent accordingly grows
like $\Lambda(\eps)^2/(2s^2)$, with a deterministic constant $\Lambda(\eps)$
that is introduced below and follows from the adversarial setting studied in
\cite{gosgens}.
The stochastic harmonic schedule
\eqref{eq:harmonic} was found by optimizing the supermartingale certificate of
Proposition~\ref{prop:cert}, whereas the adversarial analysis of \cite{gosgens}
found the schedule \eqref{eq:gvpavg} of a similar shape by optimizing against
deterministic budget restricted corruption. That two such seemingly unrelated
design problems should have solutions of the same shape is, in light of the previous discussion, no accident. At a fixed horizon the adversarial setting is precisely the small--noise limit of the stochastic one, as Proposition~\ref{prop:limit} below makes exact.

Consider any {deterministic} algorithm in the oracle
model of Section~\ref{sec:intro}. In particular, its output $x_A$ is a deterministic measurable
function of the observed subgradients $\tilde g_0,\dots,\tilde g_{N-1}$. Given the
instance and the oracle's selection, the output and its gap are then functions of
the noise path $e=(e_0,\dots,e_{N-1})$. Write $G_f(e)$ for the resulting
suboptimality gap, $\|e\|^2:=\sum_{k=0}^{N-1}\|e_k\|^2$ for the energy of a
path, and the worst gap over the instance class and all corruption paths of a
given energy as
\[
  G(\gamma)\;:=\;\sup\big\{G_f(e):\ f\in\mathcal C,\ \|e\|\le\gamma\big\},
\]
nondecreasing in the budget $\gamma$. For the averaged subgradient methods
of \cite{gosgens} this worst case admits a convex performance optimization
characterization. The {guarantee threshold}
\[
  \Lambda(\eps)\;:=\;\inf\big\{\|e\|:\ G_f(e)\ge\eps\ \text{for some }f\in\mathcal C\big\}
\]
is the least energy of a corruption path which drives the algorithm to
failure at level $\eps$. Two consequences of the definition are all that is
used below. First, every failing path has energy at least $\Lambda(\eps)$. Second,
$G(\gamma)<\eps$ forces $\Lambda(\eps)\ge\gamma$, whereas a failing path of
energy at most $\gamma$ gives $\Lambda(\eps)\le\gamma$. Like $G(\gamma)$ and
the exponent $\mathcal E_N$, the threshold $\Lambda(\eps)$ is a property of
the algorithm under consideration.
The bridge to the stochastic model is that the class \eqref{eq:subg} contains
laws under which a {prescribed deterministic corruption path $e^\star$ is realized
exactly}, and does so, for $s$ small, with probability nearly
$\exp\big(-\sum_{k=0}^{N-1}\|e^\star_k\|^2/(2s^2)\big)$. That is, no path is
impossible; the noise level $s$ merely dictates how improbable a path of a given
energy is.

\begin{lemma}\label{lem:atom}
Let $\mu\in\R^d\setminus\{0\}$, $\eta\in(0,1]$ and $\varsigma>0$, and let
$B_\eta$ be the constant of Lemma~\ref{lem:atomunit}. For every
$\varsigma\le\|\mu\|/B_\eta$ some distribution
$P_\mu$ on $\R^d$ with
\[
  \E_{P_\mu}[e]=0
  \qquad\text{and}\qquad
  \E_{P_\mu}\big[\exp\langle a,e\rangle\big]\le\exp\Big(\frac{\varsigma^2\|a\|^2}{2}\Big)
  \quad\text{for all }a\in\R^d
\]
has
\[
  P_\mu(\{\mu\})\;\ge\;e^{-(1+\eta)\|\mu\|^2/(2\varsigma^2)} .
\]
\end{lemma}

\begin{proof}
Set $\nu:=\mu/\varsigma$, so that $\|\nu\|\ge B_\eta$ by the hypothesis
$\varsigma\le\|\mu\|/B_\eta$. Lemma~\ref{lem:atomunit} applies to $\nu$ and
supplies a law $P_\nu$; let $P_\mu$ be the law of $\varsigma Y$ with
$Y\sim P_\nu$. Its mean vanishes, and for every $a\in\R^d$
\[
  \E\big[\exp\langle a,\varsigma Y\rangle\big]
  \;=\;\E\big[\exp\langle\varsigma a,Y\rangle\big]
  \;\le\;e^{\|\varsigma a\|^2/2}\;=\;e^{\varsigma^2\|a\|^2/2} ,
\]
and
$P_\mu(\{\mu\})=P_\nu(\{\nu\})\ge e^{-(1+\eta)\|\nu\|^2/2}=e^{-(1+\eta)\|\mu\|^2/(2\varsigma^2)}$.

\begin{lemma}\label{lem:atomunit}
Let $\eta\in(0,1]$. There is a constant $B_\eta\ge1$ depending only on $\eta$
such that for every $\nu\in\R^d$ with $\|\nu\|\ge B_\eta$ some
distribution $P_\nu$ on $\R^d$ with
\[
  \E_{P_\nu}[e]=0
  \qquad\text{and}\qquad
  \E_{P_\nu}\big[\exp\langle a,e\rangle\big]\le\exp\Big(\frac{\|a\|^2}{2}\Big)
  \quad\text{for all }a\in\R^d
\]
has
\[
  P_\nu(\{\nu\})\;\ge\;e^{-(1+\eta)\|\nu\|^2/2} .
\]
\end{lemma}

\begin{proof}
Write $x_+:=\|\nu\|$, the hypothesis reading $x_+\ge B_\eta\ge1$; in
particular $\nu\ne0$.

Let $X$ be a mean--zero univariate random variable with $\E[e^{tX}]\le e^{t^2/2}$
for every $t\in\R$, and take $P_\nu$ to be the law of $X\nu/x_+\in\R^d$. The
distribution $P_\nu$ is supported on the line $\R\nu$, it is mean--zero,
$P_\nu(\{\nu\})=\Prob(X=x_+)$, and for every $a\in\R^d$, writing
$t:=\langle a,\nu/x_+\rangle$ so that $|t|\le\|a\|$ by Cauchy--Schwarz since
$\nu/x_+$ is a unit vector,
\[
  \E\big[\exp\langle a,X\nu/x_+\rangle\big]=\E[e^{tX}]\le e^{t^2/2}\le
  e^{\|a\|^2/2} ,
\]
so $P_\nu$ meets both requirements of the lemma.

It remains to construct $X$, which we take supported on two points
$\{x_+,x_-\}$ with $x_-<0<x_+$. Such an $X$ is essentially determined by $x_+$
alone. Indeed, we demand $p:=e^{-(1+\eta)x_+^2/2}$ at $x_+$ and a vanishing
mean then forces the remaining mass $1-p$ to sit at $x_-:=-px_+/(1-p)$.
Choose now $B_\eta\ge1$, depending on $\eta$ alone, so large that
\begin{equation}\label{eq:atomcond}
  8px_+^2\;\le\;1
  \qquad\text{and}\qquad
  \frac1{2x_+^2}\;\ge\;-\log\big(1-e^{-\eta x_+^2/2}\big)
\end{equation}
hold whenever $x_+\ge B_\eta$. This is possible since, $p$ being
$e^{-(1+\eta)x_+^2/2}$, the left--hand side of the first condition tends to
zero, while in the second the right--hand side vanishes exponentially in
$x_+^2$ and the left--hand side only polynomially. 

The first condition in \eqref{eq:atomcond} gives $px_+^2\le1/8$ and, as $x_+\ge B_\eta\ge1$, also
$p\le1/(8x_+^2)\le1/8$. In particular $1-p\ge1/2$, so that
\[
  |x_-|\;=\;\frac{px_+}{1-p}\;\le\;2px_+ .
\]
\textbf{Case $|t|x_+\le1$.} We have $|tx_-|\le2p|t|x_+\le2p\le1$, so the
bound $e^u\le1+u+u^2$, valid for $|u|\le1$, applies at both atoms and
\[
  \E[e^{tX}]\;\le\;p\big(1+tx_++t^2x_+^2\big)+(1-p)\big(1+tx_-+t^2x_-^2\big)
  \;=\;1+t^2\big(px_+^2+(1-p)x_-^2\big),
\]
the linear term vanishing with the mean. As $x_-^2\le4p^2x_+^2\le px_+^2$ by
$4p\le1$, and $1-p\le1$, we conclude with $1+u\le e^u$ and $2px_+^2\le1/4$
that
\[
  \E[e^{tX}]\;\le\;1+2px_+^2t^2\;\le\;e^{2px_+^2t^2}\;\le\;e^{t^2/2}.
\]

\noindent\textbf{Case $tx_+\ge1$.} Completing the
square,
\[
  p\,e^{tx_+}\;=\;\exp\Big(\tfrac{t^2}2-\tfrac12\big(t-x_+\big)^2
  -\tfrac{\eta x_+^2}2\Big)\;\le\;e^{t^2/2}\,e^{-\eta x_+^2/2},
\]
while $(1-p)e^{tx_-}\le1$ as $t\ge0>x_-$, so
$\E[e^{tX}]\le e^{t^2/2}\big(e^{-\eta x_+^2/2}+e^{-t^2/2}\big)$, and the
bracket is at most one because $t^2/2\ge1/(2x_+^2)$ and the
second condition of \eqref{eq:atomcond}.

\noindent\textbf{Case $tx_+\le-1$.} $pe^{tx_+}\le p$ as $t\le0<x_+$,
while the other atom has $tx_-=|t|\,|x_-|\le2px_+|t|$. Hence
$\E[e^{tX}]\le p+(1-p)e^{2px_+|t|}\le p+(1-p)e^{t^2/2}\le e^{t^2/2}$, using
$2px_+\leq 4px_+\le|t|/2$, valid since $|t|\ge1/x_+$ and $4px_+^2\le1/2$.
\end{proof}
\end{proof}

With these laws in hand, the guarantee threshold characterizes the failure
exponent of any deterministic algorithm in the small--noise regime.

\begin{proposition}\label{prop:limit}
Fix $N$, $\eps\in(0,RL)$ and a deterministic algorithm with $0<\Lambda(\eps)<\infty$, and
write $\mathcal E_N(s)$ to display the dependence on the noise level. Then
\[
  \lim_{s\downarrow0}\ s^2\,\mathcal E_N(s)\;=\;\frac{\Lambda(\eps)^2}{2}\,.
\]
\end{proposition}

\begin{proof}

\emph{Lower.} By definition of $\Lambda(\eps)$ every failing path has energy
at least $\Lambda(\eps)$, that is,
$\{G_f(e)\ge\eps\}\subseteq\{\|e\|^2\ge\Lambda(\eps)^2\}$ for every $f$. Lemma~\ref{lem:subg} at $a=0$ bounds each conditional moment,
$\E[\exp(b\|e_k\|^2)\mid\mathcal F_k]\le(1-2bs^2)^{-d/2}$ for $0\le b<1/(2s^2)$,
uniformly over the class. Writing $S_k:=\sum_{m<k}\|e_m\|^2$, the partial sum
$S_k$ is $\mathcal F_k$--measurable, so the tower property and this bound give
\[
  \E\big[\exp(bS_{k+1})\big]
  \;=\;\E\Big[\exp(bS_k)\,\E\big[\exp\big(b\|e_k\|^2\big)\,\big|\,\mathcal F_k\big]\Big]
  \;\le\;\big(1-2bs^2\big)^{-d/2}\,\E\big[\exp(bS_k)\big]
\]
for $k=0,\dots,N-1$. Iterating from $S_0=0$ yields
\[
  \E\Big[\exp\Big(b\sum_{k=0}^{N-1}\|e_k\|^2\Big)\Big]\;\le\;\big(1-2bs^2\big)^{-Nd/2}.
\]
As $b\ge0$, the event $\{\sum_{k=0}^{N-1}\|e_k\|^2\ge\Lambda(\eps)^2\}$ is
contained in $\{\exp(b\sum_{k=0}^{N-1}\|e_k\|^2)\ge\exp(b\Lambda(\eps)^2)\}$,
so Markov's inequality  gives
\[
  \Prob\Big(\sum_{k=0}^{N-1}\|e_k\|^2\ge\Lambda(\eps)^2\Big)
  \;\le\;e^{-b\Lambda(\eps)^2}\,\E\Big[\exp\Big(b\sum_{k=0}^{N-1}\|e_k\|^2\Big)\Big]
  \;\le\;e^{-b\Lambda(\eps)^2}\big(1-2bs^2\big)^{-Nd/2}.
\]
When $\Lambda(\eps)^2\ge Nds^2$ an admissible choice is
$b=\big(1-Nds^2/\Lambda(\eps)^2\big)/(2s^2)$ and guarantees
\[
  \Prob\Big(\sum_{k=0}^{N-1}\|e_k\|^2\ge\Lambda(\eps)^2\Big)
  \;\le\;\exp\Big(-\frac{\Lambda(\eps)^2}{2s^2}+\frac{Nd}2+\frac{Nd}2\log\frac{\Lambda(\eps)^2}{Nds^2}\Big).
\]
Hence, for $s\downarrow 0$ we have
\[
  s^2\mathcal E_N(s)\;\ge\;\frac{\Lambda(\eps)^2}2-\frac{Nds^2}2\Big(1+\log\frac{\Lambda(\eps)^2}{Nds^2}\Big)
  \;\longrightarrow\;\frac{\Lambda(\eps)^2}2.
\]

\emph{Upper.} Let $\gamma>\Lambda(\eps)$. As $\Lambda(\eps)$ is an infimum,
there are a path $e^\star$ of energy $\|e^\star\|^2\le\gamma^2$ and an
instance $f\in\mathcal C$ with $G_f(e^\star)\ge\eps$. Fix $\eta>0$ and draw the noise independently as
$e_k\sim P_{e^\star_k}$ from Lemma~\ref{lem:atom} at $\varsigma=s$ or deterministically $e_k\equiv 0$ where $e^\star_k=0$.

For
$s\le\min\{\|e^\star_k\|:e^\star_k\ne0\}/B_\eta$ the lemma applies at every
$k$ with $e^\star_k\ne0$, and independence collapses the conditional expectation in
\eqref{eq:subg} to the marginal,
$\E[\exp\langle a,e_k\rangle\mid\mathcal F_k]=\E[\exp\langle a,e_k\rangle]
\le\exp(s^2\|a\|^2/2)$, so this law lies in the class. Moreover, by independence
\[
  \Prob\big(e=e^\star\big)=\prod_k \Prob\big(e_k=e_k^\star\big)\;\ge\;\prod_k
  e^{-(1+\eta)\|e^\star_k\|^2/(2s^2)}\;\ge\;e^{-(1+\eta)\gamma^2/(2s^2)} .
\]
On the event $\{e=e^\star\}$ the run suffers gap at least
$\eps$, so $\mathcal E_N(s)\le(1+\eta)\gamma^2/(2s^2)$. Let $\eta\downarrow0$,
then $\gamma\downarrow\Lambda(\eps)$.
\end{proof}

Proposition~\ref{prop:limit} reduces the small--noise exponent of a method to
its guarantee threshold, that is, to the least energy of a corruption path
which makes it fail. Everything that remains is therefore a question about the adversarial
setting alone. We first bound the threshold of the adversarial harmonic
schedule from below, then the best threshold of any method from above.

\begin{lemma}\label{lem:uniform}
Let $RL/\sqrt{N+1}\le\eps'<\eps$. The adversarial harmonic schedule
\eqref{eq:gvpavg} matched at $\eps'$ has guarantee threshold
\[
  \Lambda(\eps)\;\ge\;L\,u^{-1}\Big(\frac{\eps'\sqrt{N+1}}{RL}\Big).
\]
\end{lemma}

\begin{proof}
Let $x_A$ be the uniform average of the iterates of
\eqref{eq:method} run with any steps $h_k\ge0$, and let $\alpha_k=h_k(N-k)/(N+1)$ be
its conic combination. Jensen's inequality applied to the uniform average gives
$f(x_A)-f_\star\le\frac1{N+1}\sum_k\big(f(x_k)-f_\star\big)$, and the
subgradient inequality between $x_k$ and $x_\star$ bounds each term by
$\langle g_k,x_k-x_\star\rangle$, so that
$f(x_A)-f_\star\le\frac1{N+1}\sum_k\langle g_k,x_k-x_\star\rangle$. Substituting
$x_k=x_0-\sum_{j<k}h_j(g_j+e_j)$ with $h_j/(N+1)=\alpha_j/(N-j)$,
\[
  f(x_A)-f_\star\;\le\;\frac1{N+1}\sum_{k=0}^{N}\langle g_k,x_0-x_\star\rangle
  \;-\!\!\sum_{0\le j<k\le N}\!\!\frac{\alpha_j}{N-j}\,\langle g_k,g_j+e_j\rangle
\]
for every $f\in\mathcal C$ and every error path \cite[Equation (24)]{gosgens}.
The right--hand side involves the schedule only through its conic combination
$\alpha$, and so does its worst case over $\|x_0-x_\star\|\le R$,
$\|g_k\|\le L$ and $\sum_{k=0}^{N-1}\|e_k\|^2\le\gamma^2$, the performance
estimate of \cite[Prop.~1]{gosgens}: two methods sharing a conic combination
share this bound on $G(\gamma)$.

Write $\sigma':=u^{-1}\big(\eps'\sqrt{N+1}/(RL)\big)$ for the budget matched
at $\eps'$. By its construction in Section~\ref{sec:intro}, the schedule
\eqref{eq:gvpavg} carries exactly the conic combination that
\cite[Lem.~2]{gosgens} tunes at $\sigma'$, and for which that lemma certifies
a gap of at most $RL\,u(\sigma')/\sqrt{N+1}=\eps'$ against every error path of
energy at most $L^2\sigma'^2$. Hence $G(L\sigma')\le\eps'<\eps$: no path of
energy at most $L^2\sigma'^2$ causes failure at level $\eps$, and
$\Lambda(\eps)\ge L\sigma'$.
\end{proof}

Every
deterministic algorithm satisfying the cone condition of \cite{fatkhullin},
\begin{equation}\label{eq:cone}
  x_k\;\in\;x_0-\mathrm{cone}\big(\tilde g_0,\dots,\tilde g_{k-1}\big)
  \quad\forall k\in[1,N],
  \qquad
  x_A\;\in\;x_0-\mathrm{cone}\big(\tilde g_0,\dots,\tilde g_{N-1}\big),
\end{equation}
suffers at every $\sigma\in[0,\sqrt N\,]$ a gap of at least
$G(L\sigma)\ge(1-\delta_N)RL\,u(\sigma)/\sqrt{N+1}$ on some instance and some error path
of energy at most $L^2\sigma^2$ \cite[Thm.~2 and Cor.~2]{gosgens}, with $u$ and $\delta_N$ as in the
introduction. The condition \eqref{eq:cone} is satisfied by the vast                                                               
majority of first--order methods, including every subgradient method with                                                                    
nonnegative steps as well as Nesterov's accelerated method and most variable--stepsize schemes \cite{gosgens}.
A notable exception in the context of saddle-point problems would be the recent work \cite{shugart2025negative}.

\begin{lemma}\label{lemma:small-noise-lb}
Let $N$ be large enough that $RL/\sqrt{N+1}\le\eps\le(1-\delta_N)RL$.
The best threshold $\Lambda^\star(\eps):=\sup\Lambda(\eps)$, the supremum
running over all deterministic algorithms satisfying \eqref{eq:cone}, obeys
\[
  \Lambda^\star(\eps)\;\le\;L\,u^{-1}\Big(\frac{\eps\sqrt{N+1}}{(1-\delta_N)RL}\Big).
\]
\end{lemma}

\begin{proof}
Fix such an algorithm and a $\sigma\in[0,\sqrt N\,]$ with
$u(\sigma)\ge\eps\sqrt{N+1}/((1-\delta_N)RL)$. Then some instance and some
error path of energy at most $L^2\sigma^2$ produce a gap of at least $\eps$,
whence $\Lambda(\eps)\le L\sigma$. Let $\sigma$
decrease to $u^{-1}\big(\eps\sqrt{N+1}/((1-\delta_N)RL)\big)$, admissible as
$u$ is increasing and continuous, and take the supremum over algorithms.
\end{proof}

From Lemma \ref{lem:uniform} and continuity of $u^{-1}$, the adversarial harmonic schedule matched at $\eps'\uparrow\eps$ 
guarantees a budget $L\,u^{-1}\big(\eps\sqrt{N+1}/(RL)\big)$. The upper bound of Lemma~\ref{lemma:small-noise-lb}
exceeds this limit by a factor $1+O(\log N/N)$ only. Indeed, write
$y_N:=\eps\sqrt{N+1}/(RL)$ and $c_N:=(1-\delta_N)^{-1}\geq 1$, so that $y_N\ge2$ for all
$N$ large enough. Read in terms of the inverse, the defining relation
\eqref{eq:udef} states that $u^{-1}(v)^2=v^2-1-2\log v$ for every $v\ge1$.
Evaluating at $v=c_Ny_N$ and at $v=y_N$, both admissible as $c_Ny_N\ge y_N\ge1$, and
subtracting,
\begin{align*}
  u^{-1}(c_Ny_N)^2-u^{-1}(y_N)^2
  = & \big(c_N^2y_N^2-1-2\log y_N-2\log c_N\big)-\big(y_N^2-1-2\log y_N\big)\\
  = & (c_N^2-1)\,y_N^2-2\log c_N .
\end{align*}
The last term is nonpositive as $c_N\ge1$, so that
$u^{-1}(c_Ny_N)^2\le u^{-1}(y_N)^2+(c_N^2-1)\,y_N^2$. We may use $y_N^2\le\tfrac52\,u^{-1}(y_N)^2$ for $y_N\ge2$. Indeed,
substituting $u^{-1}(y_N)^2=y_N^2-1-2\log y_N$ this is
$y_N^2\le\tfrac52\,y_N^2-\tfrac52-5\log y_N$ which holds on $[2,\infty)$. Hence
\[
  u^{-1}(c_Ny_N)^2\;\le\;u^{-1}(y_N)^2\Big(1+\tfrac52\,(c_N^2-1)\Big),
\]
and hence  $u^{-1}(c_Ny_N)\;\le\;\big(1+O(\log N/N)\big)\,u^{-1}(y_N)$ as $c_N^2-1=O(\delta_N)=O(\log N/N)$.
The adversarial harmonic schedule is hence optimal among
methods satisfying \eqref{eq:cone} up to a relative $O(\log N/N)$ in guaranteed budget and, by
Proposition~\ref{prop:limit}, in small--noise rate.

\section{Numerical illustration}\label{sec:numerics}

Figure~\ref{fig:rates} reports a small numerical experiment on a member of the resisting oracle class used throughout,
$$f(x)=(\eps/R)\,|x-R|$$ with $x_0=0$, $x_\star=R$ and independent Gaussian noise. We estimate the
failure probability $\Prob\big(f(x_A)-f_\star\ge\eps\big)$ by Monte Carlo
($4\times10^6$ runs per point; $R=L=s=1$, $\eps=0.1$) for the stochastic harmonic
schedule \eqref{eq:harmonic}, the expectation--optimal constant steps \eqref{eq:expectation-optimal-steps}, the matched constant steps \eqref{eq:liu-steps}, and the anytime steps $h_k=R/(M\sqrt{k+1})$ of Lemma~\ref{lem:anytime}. On a
logarithmic scale the failure probability of the stochastic harmonic schedule decays
linearly in $N$ at a rate close to $\eps^2/(2R^2s^2)$, remaining below the
probability certified by Theorem~\ref{thm:main} and running, at the larger
horizons, roughly parallel to the worst--case floor of Theorem~\ref{thm:universal}. The two constant schedules decay at visibly
flatter rates and run close together over this horizon.
The anytime schedule decays at the flattest rate.

\begin{figure}[ht]
\centering
\begin{tikzpicture}
\begin{axis}[
  width=0.65\textwidth, height=0.5\textwidth,
  xlabel={$N$},
  ylabel={$\widehat{\Prob}\big(f(x_A)-f_\star\ge\eps\big)$},
  ymode=log,
  legend cell align=left,
  legend style={at={(1.03,0.5)}, anchor=west, draw=black, fill=none, font=\small},
  xmin=200, xmax=2100,
  xtick={0,500,1000,1500,2000},
  grid=major, grid style={dotted, gray!40},
]
\addplot+[thick, mark=*]         table[x=N, y expr=exp(-\thisrow{harmonic})] {julia/rates.dat};
\addplot+[thick, mark=square*]   table[x=N, y expr=exp(-\thisrow{constant})] {julia/rates.dat};
\addplot+[thick, mark=diamond*]  table[x=N, y expr=exp(-\thisrow{liu})]      {julia/rates.dat};
\addplot+[thick, mark=triangle*] table[x=N, y expr=exp(-\thisrow{anytime})]  {julia/rates.dat};
\addplot[dashed, thick]          table[x=N, y expr=exp(-\thisrow{ceiling})]  {julia/rates.dat};
\addplot[dotted, thick, restrict expr to domain={\thisrow{cert}}{0.001:1000}]
                                 table[x=N, y expr=exp(-\thisrow{cert})]     {julia/rates.dat};
\legend{stochastic harmonic \eqref{eq:harmonic}, constant steps \eqref{eq:expectation-optimal-steps}, constant steps \eqref{eq:liu-steps},
        anytime steps,
        floor (Thm.~\ref{thm:universal}), certificate (Thm.~\ref{thm:main})}
\end{axis}
\end{tikzpicture}
\caption{Monte Carlo failure probabilities (logarithmic scale) on our resisting
instance $f(x)=(\eps/R)\,|x-R|$ ($d=1$, $R=L=s=1$, $\eps=0.1$, $4\times10^6$ runs
per point) for the four schedules discussed in the paper, against the probability
certified by Theorem~\ref{thm:main} and the universal floor of
Theorem~\ref{thm:universal}.}
\label{fig:rates}
\end{figure}
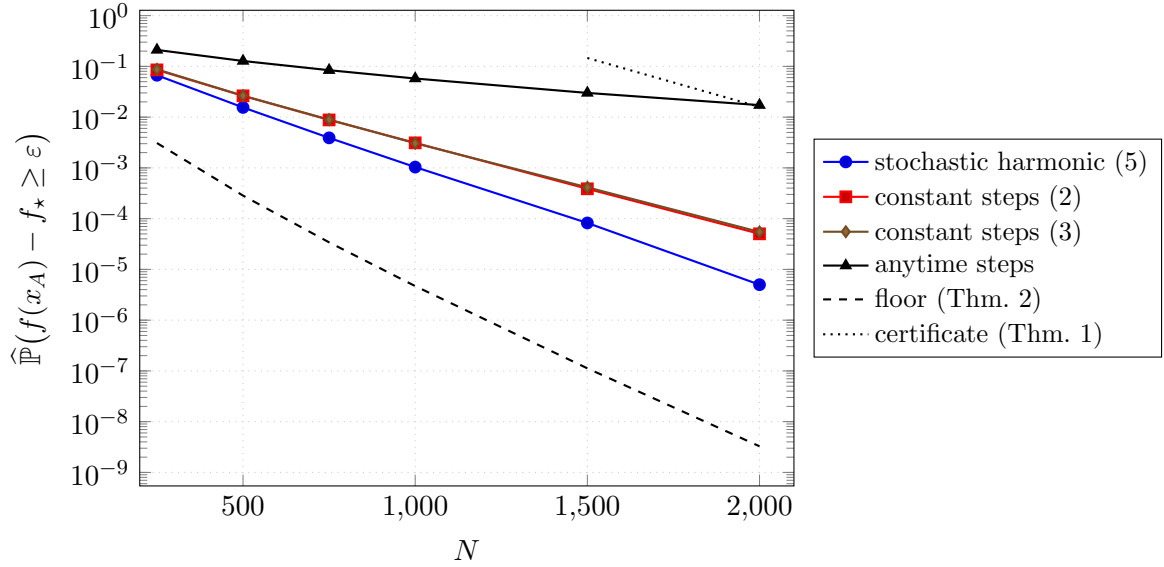

\section{Conclusion}\label{sec:conclusion}

What should a subgradient method look like when its performance is measured by how
fast the probability that its suboptimality gap exceeds a tolerance $\eps$ decays
with the iteration number? For nonsmooth convex functions with
bounded subgradients corrupted with sub--Gaussian noise, the answer is the
stochastic harmonic schedule. Its failure probability
decays asymptotically at the exponential rate $\mathcal E_N = \eps^2N/(2R^2s^2)(1+o(1))$, which is optimal among all algorithms, whether adaptive or randomized.
In the small--noise limit $s\downarrow0$, finally, the problem collapses to the adversarial model studied by \cite{gosgens}. At fixed iteration number the best achievable failure exponent grows as
$\Lambda^\star(\eps)^2/(2s^2)$, where $\Lambda^\star(\eps)$ is the largest corruption budget
at which the adversarial version of the problem still admits an algorithm
keeping the suboptimality gap below $\eps$. The adversarial harmonic schedule of
\cite{gosgens} attains this budget up to a relative $O(\log N/N)$ loss among all methods satisfying
a natural cone condition \cite{gosgens,fatkhullin}.

The present work opens several directions for future research.
Within the nonsmooth setting considered here, a noise model in which the variance scale
$s$ of the noise model \eqref{eq:subg} grows with the size of the gradient is natural.
The associated analysis of convex smooth optimization also remains open. In the absence of noise, the optimal method there uses momentum and is indeed outside our memoryless iterate class \eqref{eq:method}.

\section*{Acknowledgments}

Bart P.G.\ van Parys gratefully acknowledges funding from NWO Vidi grant VI.Vidi.243.021.

\section*{AI Use Statement}

Claude (Anthropic) was used in the preparation of this manuscript for copy editing (improving exposition) and
brainstorming (exploring proof strategies and the numerical experiment).
All results, proofs and their verification are the responsibility of the author.

\clearpage

\bibliographystyle{plain}
\bibliography{main}

\end{document}